\documentclass[a4paper]{article}
\usepackage{amsfonts}
\usepackage{amssymb,amsthm, amsmath,graphicx,bm,ebezier,amsthm}
\usepackage{hyperref}
\usepackage{caption}
\usepackage{url,lineno}
\usepackage[utf8]{inputenc}
\usepackage{arydshln}
\usepackage[ruled,vlined,linesnumbered]{algorithm2e}
\usepackage[noend]{algpseudocode}
\usepackage{lscape}
\usepackage{multirow}
\usepackage[dvipsnames]{xcolor}
\usepackage{stmaryrd}

\newtheorem{theorem}{Theorem}

\newtheorem{proposition}[theorem]{Proposition}
\newtheorem{corollary}[theorem]{Corollary}
\newtheorem{lemma}[theorem]{Lemma}
\newtheorem{conjecture}[theorem]{Conjecture}
\theoremstyle{remark}

\newtheorem{example}[theorem]{Example}
\newtheorem{remark}[theorem]{Remark}

\def\CaN{\mathcal{N}}
\def\CaH{\mathcal{H}}
\def\CaC{\mathcal{C}}

\def\CaB{\mathcal{B}}
\def\CaF{\mathcal{F}}

\def\N{\mathbb{N}}

\def\Z{\mathbb{Z}}

\def\msg{\text{msg}}
\def\g{\text{g}}
\def\m{\text{m}}
\def\e{\text{e}}
\def\c{\text{c}}
\def\n{\text{n}}
\def\se{\text{N}}
\def\Fb{\mathrm{Fb}}
\def\mult{\text{mult}}

\title{A computational approach to the study of finite-complement submonids of an affine cone}

\date{}
\author{
J. C. Rosales, 
R. Tapia-Ramos,
and A. Vigneron-Tenorio
}

\begin{document}

\maketitle

\abstract{
Let $\mathcal{C}\subseteq \mathbb{N}^p$ be an integer cone. A $\mathcal{C}$-semigroup $S\subseteq \mathcal{C}$ is an affine semigroup such that the set $\mathcal{C}\setminus S$ is finite. Such $\mathcal{C}$-semigroups are central to our study. We develop new algorithms for computing $\mathcal{C}$-semigroups with specified invariants, including genus, Frobenius element, and their combinations, among other invariants. To achieve this, we introduce a new class of $\mathcal{C}$-semigroups, termed $\mathcal{B}$-semigroups. By fixing the degree lexicographic order, we also research the embedding dimension for both ordinary and mult-embedded $\mathbb{N}^2$-semigroups. These results are applied to test some generalizations of Wilf's conjecture.}

{\small


}

\section*{Introduction}

Let $\N$ be the set of natural numbers. We consider an affine semigroup $S$ to be a finitely generated commutative additive submonoid of $\N^p$ (for a positive integer $p$) such that the zero element belongs to $S$. For convenience, we use $0$ instead of $(0,0,\ldots,0)\in\N^p$ whenever it is unambiguous. It is well known that any affine semigroup $S$ admits a unique minimal system of generators, denoted by $\msg(S)$ (see \cite{libroR-G99}), and the cardinality of the minimal generating set, called the embedding dimension, is represented by $\e(S)$. Let $\CaC \subseteq \N^p$ be an affine (integer) cone. A submonoid $S\subseteq\CaC$ is a $\CaC$-semigroup if the set $\CaC\setminus S$ is finite; this structure was introduced in \cite{Wilf}. When $\CaC = \N^p$, $S$  is referred to as a generalized numerical semigroup, first defined in \cite{F-P-U-AlgGNS}. In the special case where $p=1$, the $\CaC$-semigroup $S$ is known as a numerical semigroup. Note that the structures of $\CaC$-semigroups and generalized numerical semigroups naturally extend the notion of numerical semigroups to higher dimensions.

Most of the invariants analysed in the study of numerical semigroups can be generalized to $\CaC$-semigroups. In addition to the embedding dimension, the set $\CaH(S)=\CaC \setminus S$ is called the set of gaps of $S$, and the genus, $\g(S)$, is the cardinality of $\CaH(S)$. Following the notation given in \cite{JI-A-R-MAX}, let $\{\tau_1, \ldots, \tau_t\}$ be the set of extremal rays of $\CaC$. For each $i \in \{1, \ldots, t\}$, the i-multiplicity of $S$, denoted by $\mult_i(S)$, is the minimum element in $\tau_i \cap S$ under the componentwise partial order in $\N^p$. To extend certain invariants to $\CaC$-semigroups, it is necessary to define a total order on $\N^p$, which is an order relation $\preceq$ on $\N^p$ that is compatible with addition and satisfies $0 \preceq x$ for any $x\in \N^p$ (see \cite{Cox}). Once a total order $\preceq$ is fixed on $\N^p$, for instance, the Frobenius element of $S$, $\Fb(S)$, is defined as $\max_{\preceq}(\CaC \setminus S)$. By convention, if $S=\CaC$, then $\Fb(S)=(-1,-1,\ldots, -1)$. The conductor of $S$, denoted by $\c(S)$, is the minimum element $x \in S$ such that $\Fb(S) \prec  x$. An element $s$ of $S$ is said to be a small element if $s\prec \Fb(S)$. The set of all small elements is denoted by $\se(S)$, and its cardinality by $\n(S)$.  Additionally, the smallest non-zero element of $S$ with respect to the total order $\preceq$ is called the multiplicity of $S$, denoted by $m(S)$. For any element $f$ of $\CaC$, let $\CaN(f)$ the cardinality of the set $\{x\in \CaC \setminus\{0\} \mid x\preceq f\}$. In particular, when $f$ is the Frobenius element of $S$, $\CaN(\Fb(S))$ is referred to as the Frobenius number of $S$. In the case of numerical semigroups, note that $\CaN(\Fb(S)) = \Fb(S)$.

Although the study of $\CaC$-semigroups and generalized numerical semigroup is relatively recent, much research has focused on examining these structures through their invariants. For instance, \cite{CarmeloPedro} and \cite{ F-P-U-AlgGNS} include algorithms for computing all possible $\N^p$-semigroups with a fixed genus. A recent study on the unbounded behaviour of certain invariants, such as the conductor, in $\CaC$-semigroups can be found in \cite{unboundednessInvariantsCsemgp}. Moreover, \cite{SomeProperties} provides a method to compute the set of all $\CaC$-semigroups with a fixed Frobenius element, defined as $\mathfrak{C}(Fb=f)=\{S \text{ is a } \CaC\text{-semigroup} \mid \Fb(S)=f\}$, we provide an alternative procedure for computing it. 

In this work, we contribute to this ongoing research by developing and describing several algorithms to compute all possible $\CaC$-semigroups with specified invariants, including the small elements and the genus, the Frobenius element, the genus, and the combination thereof, i.e., we focus on $\mathfrak{C}(\text{gen}=g, \text{se}=n)$ which corresponds to $\CaC$-semigroups with a fixed genus and a number of small elements, $\mathfrak{C}(Fb=f, \text{gen}=g)$ for a fixed Frobenius element and genus, and $\mathfrak{C}(\text{gen}=g)$ denotes the set of all $\CaC$-semigroups with genus $g$. Additionally, we develop a new class of $\CaC$-semigroups based on their invariants, which we call $\CaB$-semigroups. We provide a graphical classification of these semigroups and show how the study of $\CaB$-semigroups is a tool for computing the set $\mathfrak{C}(Fb=f)$.

In 1978, Wilf conjectured that for any numerical semigroup $S$, the inequality $\e(S) \n(S) \geq \Fb(S) + 1$
holds (see \cite{OriginalWilf}). Although the general case of this conjecture remains unresolved, specific cases have been addressed  (see, for instance, \cite{Dobbs, EliahouWilf} and  \cite{Fromentin}). A detailed discussion of this topic can be found in \cite{DelgadoWilf} and the references therein. The suggestion to extend Wilf's conjecture to higher dimensional structures was proposed in \cite{F-P-U-AlgGNS}, leading to several contributions such as the Generalize Wilf Conjecture (see \cite{CistoWilf}) and the Extended Wilf Conjecture (see \cite{Wilf}).

Following the analysis of invariants, another objective of this work is to discuss the embedding dimension and the minimal generated set of two specific classes of $\N^2$-semigroups: ordinary $\N^2$-semigroups and mult-embedded $\N^2$-semigroups. Given a total order $\preceq$, an ordinary $\CaC$-semigroup is defined as a semigroup $S_{\preceq,c} = \{0\} \cup \{x \in \CaC \mid x \succeq c\}$ for some $c \in \CaC$. Note that our definition of an ordinary semigroup differs from the one given in \cite{CistoOrdinary}.
In the context of numerical semigroups, it is also known as \textit{half-lines} (see \cite{libroRosales}). In contrast, a mult-embedded $\CaC$-semigroup is defined as $S=\{m, 2m, \ldots, (k-1)m\}\sqcup S_{\preceq, km}$, where $m$ is a non-zero element of $\CaC$.
For $\CaC=\N^2$, studying the embedding dimension of these classes of semigroups ordered by the graded lexicographic order allows us to test the Generalized and the Extended Wilf Conjecture for $\N^2$-semigroups.

The content of this work is organized as follows: In Section \ref{2}, we study the sets $\mathfrak{C}(Fb=f, gen=g)$ and  $\mathfrak{C}(gen=g)$. We show some bounds for computing the mentioned invariants. Section \ref{Bsemgp} is devoted to study $\mathcal{B}$-semigroups. Besides, some procedures are given to compute the sets of all $\CaB$-semigroups with a fixed genus, a fixed Frobenius element, and both fixed genus and Frobenius element. In Section \ref{C-smgpFb}, we provide an algorithm to compute the set $\mathfrak{C}(Fb=f)$ by using $\CaB$-semigroups. In the last sections (Section \ref{Ordinary} and \ref{Embedded}), we prove that ordinary and mult-embedded $\N^2$-semigroups ordered by the graded lexicographic order satisfy the Generalized and the Extended Wilf Conjecture. The results introduced throughout are not only theoretical but also provide computational methods, which are illustrated through examples. To this aim, we have used third-party software (\texttt{Normaliz} \cite{Normaliz}) and some libraries developed by the authors in \texttt{Mathematica} \cite{Mathematica}.

\section{$\CaC$-semigroups with fixed Frobenius element and genus, and fixed genus}\label{2}

Given $A$ a non-empty subset of $\mathbb{R}_{\ge}^p$, with $\mathbb{R}_\geq$ denoting the set of non-negative real numbers, the real cone determined by $A$ is
$$L(A)=\left\{\sum_{i=1}^h\lambda_ia_i\mid h\in \N,\, \lambda_1,\ldots,\lambda_h\in \mathbb{R}_{\ge},\, a_1,\ldots,a_h\in A\right\},$$
and the integer cone determined by $A$ is the set $L(A)\cap \N^p$. In general, a non-degenerated real (or integer) cone is the set of real (or integer) points belonging to the convex hull of finitely many half lines in $\mathbb{R}_\ge^p$ emanating from the origin. 

An integer cone $\CaC\subseteq \N^p$ is affine, that is, it is finitely generated, if there is a finite subset $A$ of $\CaC$ such that $\CaC= L(A)\cap \N^p$. In \cite[Chapter 2]{Bruns} is proved that a cone $\CaC\subseteq \N^p$ is affine if and only if it has a rational point in each of its extremal rays. Moreover, any subsemigroup of $\CaC$ is finitely generated if and only if there exists an element in the subsemigroup in each extremal ray of $\CaC$. We assume that any integer cone considered in this work is affine.

Fix an integer cone $\CaC$, and a total order $\preceq$ on $\N^p$. Recall that 
\[\mathfrak{C}(Fb=f, gen=g)=\{S \text{ is a } \CaC\text{-semigroup } \mid \Fb(S)=f, \g(S)=g\}.\]
In this section, we have two main goals. The objectives are to describe an algorithm for the computation of $\mathfrak{C}(Fb=f, gen=g)$ and to provide a procedure to compute those $\CaC$-semigroups with a fixed genus. We consider some bounds to the genus and the Frobenius number of the semigroups to achieve these.

Using the terminology from \cite{PseudoFrobenius}, a $\CaC$-semigroup is called irreducible if it cannot be expressed as the intersection of two $\CaC$-semigroups that properly contain it. For any $x,y\in L\subseteq\N^p$, consider the partial order $x \leq_L y$  if  $y - x \in L$. The following proposition establishes an irreducible $\CaC$-semigroup from an existing $\CaC$-semigroup, we provide the most general case applying any total order. 
 
\begin{proposition}\cite[Lemma 12]{SomeProperties}
\label{EIrreducible}
    Let $S$ be a $\CaC$-semigroup and let $f\in \CaC\setminus\{0\}$. Then, the set
    \[
    \Delta(S,f)=\bigg(\CaC\setminus\{f\}\bigg)\setminus \bigg\{x\in S\setminus\{0\}\mid x\leq_\CaC f, \text{ and } x\preceq \frac{f}{2}\bigg\}
    \]
    is an irreducible $\CaC$-semigroup, with Frobenius element $f$.
\end{proposition}

The next proposition combines Corollaries 8 and 9 in \cite{SomeProperties}. We denote by $\CaB(f)$ the set $\{x\in\CaC\mid x\leq_\CaC f\}$ with $f\in \CaC\setminus\{0\}$. From now on, $\lceil \cdot\rceil$ denotes the ceiling function, which rounds up to the nearest integer, and for any set $A$, the symbol $\sharp$ denotes the cardinality of the set $A$. Besides, using the notation from \cite{P-L-Fran}, for any two natural numbers $a$ and $b$, with
$a\leq b$, the set $\llbracket a, b \rrbracket = \{r \in \N \mid a \leq r \leq b\}$. If $a=0$, instead of $\llbracket 0, b \rrbracket$, we use $\llbracket  b \rrbracket$ for short.

\begin{proposition}\label{S_irred}
        Let $S$ be a $\CaC$-semigroup with Frobenius element $f$. Then, $S$ is irreducible if and only if $\g(S)=\left\lceil \dfrac{\sharp \CaB(f)}{2} \right\rceil$.
\end{proposition}

In the specific context of $\N^p$-semigroups, the value of the cardinality of $\CaB(f)$ can be determined explicitly as $\sharp \CaB(f)=\prod_{i\in\llbracket p\rrbracket}(f_i+1)$ for $f=(f_1,\ldots, f_p)\in \CaC\setminus\{0\}$.

\begin{corollary}\cite[Theorem 5.6 and Theorem 5.7]{Cisto2019}
\label{CistoHuecos}
    Let $S$ be a $\N^p$-semigroup with Frobenius element $f=(f_1,\ldots, f_p)$. Then, $S$ is irreducible if and only if $\g(S)=\left\lceil \frac{\prod_{i\in\llbracket p\rrbracket}(f_i+1)}{2} \right\rceil$.
\end{corollary}
 
We gather the previous results to obtain the announcement result. For any element $f\in \CaC$, recall that  $\CaN(f)$ is the cardinality of the set $\{x\in \CaC\setminus\{0\}\mid x\preceq f\}$. For this section, we need a total order $\preceq$ such that $\CaN(f)$ is finite. So, we assume that the fixed total order $\preceq$ satisfies that property. For example, a graded order can be used (see \cite{Cox}).

\begin{theorem}\label{cacCFG}
For any $f\in \CaC\setminus\{0\}$ and any positive integer $g$, the set $\mathfrak{C}(Fb=f, gen=g)$ is non-empty if and only if
    \[\left\lceil \frac{\sharp \CaB(f)}{2} \right\rceil\leq g \leq \CaN(f).
    \]
\end{theorem}

\begin{proof}
 Assume that $\mathfrak{C}(Fb=f, gen=g)$ is non-empty, and let $S\in \mathfrak{C}(Fb=f, gen=g)$. Trivially, $\g(S)\leq \CaN(f)$. To the other inequality, consider $X=\CaB(f)\cap S$ and $Y=\CaB(f)\setminus S$. Define the injective map $\varphi: X\longrightarrow Y$ via $\varphi(x)=f-x$. Note that $\varphi$ is well-defined, since $f-x\notin S$, otherwise $f\in S$, which it is not possible. Therefore, $\sharp X\leq \sharp Y$, and since $\CaB(f)$ equals the disjoint union of $X$ and $Y$, then $\g (S)\geq \sharp Y\geq \big\lceil \frac{\sharp \CaB(f)}{2} \big\rceil$.

Conversely, by combining Propositions \ref{EIrreducible} and \ref{S_irred}, we obtain that $\Delta(\CaC, f)=(\CaC\setminus \{f\})\setminus\{x\in \CaC\setminus\{0\} \mid   x\leq_\CaC f, \text{ and } x\preceq \frac{f}{2}\}$ belongs to $\mathfrak{C}\left(Fb=f, gen=\big\lceil \frac{\sharp \CaB(f)}{2} \big\rceil\right)$. Hence, we can define the following sequence: $T_0=\Delta(\CaC, f)$, and $T_{i+1}=T_i\setminus\{\m(T_i)\}$ if $\m(T_i)\prec f$, otherwise $T_{i+1}=T_i$, for every positive integer $i$. Note that there exists some natural number $i_0$ such that $T_{i_0}=T_{i_0+1}$, and $\cup_{i=0}^{i_0} \g(T_i)=\llbracket\big\lceil \frac{\sharp \CaB(f)}{2} \big\rceil,\CaN(f)\rrbracket$.
\end{proof}

Since $\sharp \CaB(f)$ equals $\prod_{i\in\llbracket p\rrbracket}(f_i+1)$ when $\CaC=\N^p$, the previous result can be specialized to $\N^p$-semigroups. 

\begin{corollary}\label{corocacCFG}
For any $f\in \N^p\setminus\{0\}$ and any positive integer $g$, the set $\mathfrak{C}(Fb=f, gen=g)$ is non-empty if and only if
    \[\left\lceil \frac{\prod_{i\in\llbracket p\rrbracket}(f_i+1)}{2} \right\rceil\leq g \leq \CaN(f).
    \]
\end{corollary}

The last theorem provides an algorithm (Algorithm \ref{AlgcomputeCFrGenus}) to compute the set of all $\CaC$-semigroups fixed the Frobenius element and the genus. Before presenting the algorithm, we introduce two definitions. We denote by $\Delta(f)$ the set $\{x\in \CaC \mid x\succ f\}\cup \{0\}$. Note that $\Delta(f)$ is an ordinary $\CaC$-semigroup with Frobenius element $f$. For any $\CaC$-semigroup $S$, we say that $x\in \CaH(S)$ is a special gap of $S$ if $x+S\setminus\{0\}\subset S$, and $2x\in S$. The set of all special gaps of $S$ is denoted by $SG(S)$. 

From the argument given in the proof of Theorem \ref{cacCFG}, it follows that for any positive integer $i$, if $\m(T_i)\prec f$ then $T_{i+1}\cup \{\m(T_i)\}=T_i$. Since $T_i$ is a $\CaC$-semigroup, it follows that $\m(T_i)\in SG(T_{i+1})$. This fact, combined with the definition of $\Delta(f)$, ensures that the set $B$ in the following algorithm is non-empty.

\begin{algorithm}[H]
\caption{Computing the set $\mathfrak{C}(Fb=f, gen=g)$.}\label{AlgcomputeCFrGenus}
\KwIn{Let $f\in\CaC\setminus\{0\}$ and a positive integer $g$.}
\KwOut{The set  $\mathfrak{C}(Fb=f, gen=g)$.}
 \If {$g\notin \llbracket\big\lceil \frac{\sharp \CaB(f)}{2} \big\rceil,\CaN(f)\rrbracket$}
    {\Return{$ \emptyset $}}
$A \leftarrow \{\Delta(f)\}$\;
\For{$i\in [0,\CaN(f)-g)$}{
    $Y \leftarrow \emptyset$\;
    \While {$A\ne\emptyset$}{
        $T \leftarrow \text{First}(A)$\;
        $B \leftarrow \{x\in SG(T)\setminus\{f\}\mid x\prec \m(T)\}$\;
        $Y \leftarrow Y\cup\{T\cup\{x\}\mid x\in B\}$\;
        $A \leftarrow A\setminus \{T\}$\;
    }
    $A \leftarrow Y$\;
    }
    \Return{$A$}
\end{algorithm}

Our work aims not to perform a computational comparison between existing algorithms and the alternatives we propose. We focus on providing alternative algorithms that offer distinct approaches to the problem at hand. We illustrate Algorithm \ref{AlgcomputeCFrGenus} with the following example.

\begin{example}\label{EjFRB}
    Consider the degree lexicographic order, and let $f=(2,2)$ and $g=5$. The $\CaC$-semigroup $\Delta(f)$ is shown in Figure \ref{Deltaplot}, where the empty circles are the gaps of $\Delta(f)$, the blue squares are the minimal generators of $\Delta(f)$, and the red circles are elements of $\Delta(f)$. 
    
    \begin{figure}[ht]
    \centering
    \includegraphics[scale=.45]{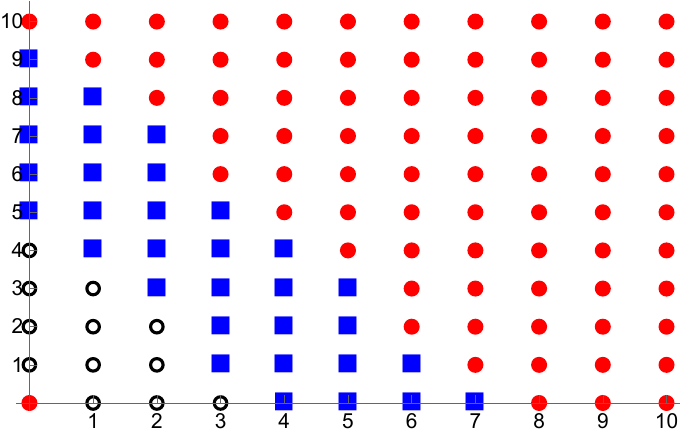}
    \caption{$\CaC$-semigroup $\Delta(f)$.}
    \label{Deltaplot}
\end{figure}
    By applying Algorithm \ref{AlgcomputeCFrGenus}, we obtain that the set $\mathfrak{C}(Fb=(2,2), gen=5)$ is
    \begin{align*}
\big\{& S_1=\left\{(0, 1), (1, 2), (2, 3), (3, 0), (4, 0), (5, 0)\right\},\\
&S_2=\left\{(0, 3), (0, 4), (0, 5), (1, 0), (2, 1), (3, 2)\right\},\\
&S_3=\left\{(0, 2), (0, 3), (1, 2), (1, 3), (2, 1), (3, 0), (3, 1), (4, 0), (4, 
1), (5, 0)\right\},\\
&S_4=\left\{(0, 3), (0, 4), (0, 5), (1, 2), (1, 3), (1, 4), (2, 0), (2, 1), (3, 
0), (3, 1)\right\}
\big\},
\end{align*}
where $S_i$ is the minimal generating set of each $\N^2$-semigroup.

From Theorem \ref{cacCFG} (or Corollary \ref{corocacCFG}), we know that fixed the Frobenius element $(2,2)$ there exist at least one $\N^2$-semigroup with genus belonging to  $\llbracket\big\lceil \frac{\sharp \CaB(f)}{2} \big\rceil,\CaN(f)\rrbracket=\llbracket 5,12\rrbracket$. Table \ref{tab1} shows the cardinality of $\mathfrak{C}(Fb=(2,2), gen=g)$ for $g$ in $\llbracket 5,12\rrbracket$.

\begin{table}[h]
    \centering
    \begin{tabular}{|c|c|c|c|c|c|c|c|c|}
    \hline
      Genus  &  5 &  6 & 7 & 8 & 9 & 10 & 11 & 12 \\ \hline
        Cardinality & 4 & 17 & 37 & 49 & 41 & 22 & 7 & 1 \\\hline
    \end{tabular}
    \caption{For $\CaC=\N^2$, $\sharp \mathfrak{C}(Fb=(2,2), gen=g)$ for all $g\in\llbracket 5,12\rrbracket$.}
    \label{tab1}
\end{table}

\end{example}

Motivated by their relationship between the genus and the number of Frobenius, precisely, $\n(S)-1+\g(S)= \CaN(\Fb(S))$,  we turn our attention to provide a method for computing the set of $\CaC$-semigroups with a fixed genus and a fixed number of small elements.

\begin{proposition}\label{propn(s)}
    Let $S$ be a $\CaC$-semigroup, $f\in \CaC\setminus\{0\}$, and let $g$ be a positive integer such that $\big\lceil\frac{\sharp \CaB(f)}{2}\big\rceil\leq g \leq \CaN(f)$. Then, $S\in \mathfrak{C}(Fb=f, gen=g)$ if and only if $\n(S)= \CaN(f)-1-g$ and $\g(S)=g$.
\end{proposition}

If $S$ is a numerical semigroup, then $\n(S)\leq \g(S)$. This inequality holds because for each small element $s$, $\Fb(S)-s$ is a gap of $S$; otherwise, $\Fb(S) = x + s$ for some $x \in S\setminus\{0\}$, which contradicts the definition of $\Fb(S)$. In contrast, for $\CaC$-semigroups, this inequality does not necessarily hold, as illustrated by the following counterexample.

\begin{example}
Let $S$ be the $\N^2$-semigroup graphically represented in Figure \ref{nota19}, where the empty circles are the gaps of $S$, the blue squares are the minimal generators of $S$, and the red circles are elements of $S$. In this example, we fix the degree lexicographic order. Note that $\g(S)=9$, $\Fb(S)=(5,1)$, and $\n(S)=17$. Hence, $\n(S)> \g(S)$.
\begin{figure}[ht]
    \centering
    \includegraphics[scale=.45]{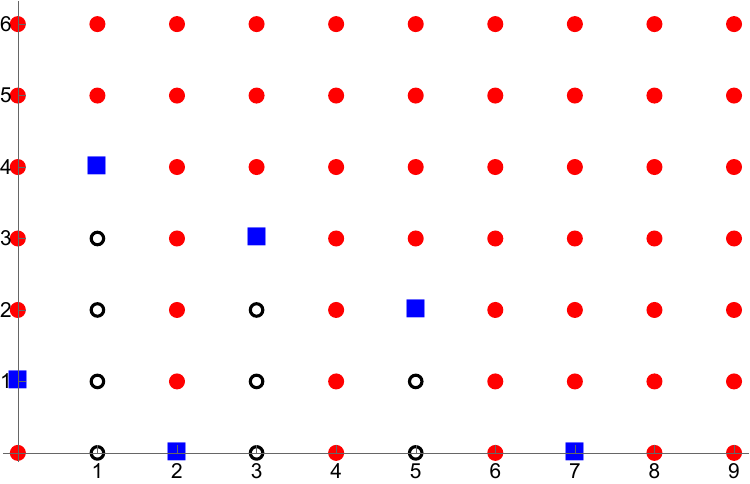}
    \caption{A $\N^2$-semigroup $S$ such that $\n(S)> \g(S)$.}
    \label{nota19}
\end{figure}
\end{example}

To achieve the announcement method, for any two natural numbers $n$ and $g$, recall that $\mathfrak{C}(\text{gen}=g, \text{se}=n)$ denotes the set of the $\CaC$-semigroups with genus $g$ and $n$ small elements. Since the genus and the number of small elements of any $\CaC$-semigroups depend on the choice of total order, the set $\mathfrak{C}(\text{gen}=g, \text{se}=n)$ does as well. We deduce the following result as a consequence of Proposition \ref{propn(s)}.

\begin{corollary}\label{CoroAgse}
    If $g$ and $n$ are two positive integers, then \[
    \mathfrak{C}(gen=g, se=n)=\mathfrak{C}(Fb=f, gen=g),\]
    where $f\in \CaC$ satisfies $\CaN(f)=g+n-1$.
\end{corollary}

Based on Corollary \ref{CoroAgse}, a method to explicitly determine $\mathfrak{C}(\text{gen}=g, \text{se}=n)$ is obtained. The first step of this algorithm is to look for the element $f$ such that $\CaN(f)=g+n-1$, and the second one is to apply Algorithm \ref{AlgcomputeCFrGenus} to get $\mathfrak{C}(Fb=f, gen=g)$. For example, taking $\CaC=\N^2$, the set $\mathfrak{C}(\text{gen}=4, \text{se}=5)$ corresponds with $\mathfrak{C}(Fb=(2,1), \text{gen}=4)$, which has already been computed (see Example \ref{EjFRB}).

Recall that $\mathfrak{C}( gen=g)$ is the set formed by all the $\CaC$-semigroups with genus equals $g$. We focus on introducing an algorithm to compute $\mathfrak{C}( gen=g)$. Note that this set is non-empty since the ordinary $\CaC$-semigroup $\Delta(f)\in \mathfrak{C}( gen=g)$, where $f\in \CaC\setminus\{0\}$ satisfies $\CaN(f)=g$. 

Given $g$ a positive integer, we define $\CaF(g)=\{\Fb(S) \mid S\in \mathfrak{C}( gen=g)\}$. Clearly, $\mathfrak{C}( gen=g)=\cup_{f\in \CaF(g)} \mathfrak{C}(Fb=f, gen=g)$. Consequently, to compute $\mathfrak{C}( gen=g)$, we proceed as follows. First, we compute the set $\CaF(g)$. Second, for each $f\in \CaF(g)$, we compute $ \mathfrak{C}(Fb=f, gen=g)$. Since Algorithm \ref{AlgcomputeCFrGenus} addresses the second step, we develop a method to compute $\CaF(g)$ without computing the set $\mathfrak{C}(gen=g)$.

\begin{proposition}\label{prosetF(g)}
    If $g$ is a positive integer, then
    \[
    \CaF(g)=\left\{f\in \CaC\mid  \left\lceil \frac{\sharp \CaB(f)}{2} \right\rceil\leq g \leq \CaN(f)\right \}.
    \]
\end{proposition}

\begin{proof}
Given $f\in \CaF(g)$, by applying Theorem \ref{cacCFG} there exists at least a $\CaC$-semigroup with genus $g$ and Frobenius element $f$.    
\end{proof}

\begin{example}\label{ejemplogenero5}
    Consider $\CaC=\N^2$. Hence, the set $\CaF(5)$ is
\begin{multline}\label{conjuntoejemplogenero5}
    \{(0, 3), (0, 4), (0, 5), (0, 6), (0, 7), (0, 8), (0, 9), (1, 2), (1, 
3), (1, 4), (2, 0), \\ (2, 1), (2, 2), (3, 0), (3, 1), (4, 0), (4, 1), 
(5, 0), (6, 0), (7, 0), (8, 0), (9, 0)\}.
\end{multline}
It is well known that $\sharp\mathfrak{C}(gen=5)=210$ (see \cite[Table 3]{Wilf}).

\end{example}

\section{$\CaB$-semigroups}\label{Bsemgp}
Fixed $f\in \CaC\setminus\{0\}$, and a total order $\preceq$ on $\N^p$, for any $\CaC$-semigroup $S$, we consider $\mathcal{O}(S)=S\cup(\CaC\setminus \CaB(f))$, recall that $\CaB(f)=\{x\in\CaC\mid x\leq_\CaC f\}$ with $f\in \CaC\setminus\{0\}$. It is straightforward from the definition that $\mathcal{O}(S)$ is a $\CaC$-semigroup. In particular, if  $\mathcal{O}(S)=S$ we say that $S$ is a $\CaB$-semigroup. The set of all
$\CaB$-semigroups with Frobenius element $f$ is denoted by $\mathfrak{B}(Fb=f)$.

This section is devoted to discuss $\CaB$-semigroups. This study is mainly structured as in Section \ref{2}. Firstly, we compute the set of $\mathfrak{B}(Fb=f)$. To achieve this, we provide an algorithm which additionally allows us to introduce its associative tree, whose vertex set is $\mathfrak{B}(Fb=f)$. Secondly, we compute the set  $\mathfrak{B}(Fb=f, gen=g)$, that is, the set of all $\CaB$-semigroups with Frobenius element $f$ and genus $g$. And finally, we compute the set of all $\CaB$-semigroups with genus $g$, denoted by $\mathfrak{B}(gen=g)$.

For any $\CaC$-semigroup $S$, we consider 
\[\alpha(S)=\min_\preceq\{x\in S\setminus\{0\} \mid x\in \CaB(\Fb(S))\}.\]
This element can be interpreted as the multiplicity in $\CaB(\Fb(S))$ of $S$. If $S\cap \CaB(f)=\{0\}$, we consider that $\alpha(S)=f$.

In this context, for any $f\in \CaC\setminus\{0\}$, we define the graph $G(\mathfrak{B}(Fb=f))$, whose vertex set is $\mathfrak{B}(Fb=f)$ and the pair $(S,T)\in \mathfrak{B}^2(Fb=f)$ is an edge if and only if $T=S\setminus \{\alpha(S)\}$. If $(S,T)$ is an edge, we say that $S$ is a child of $T$. A path connecting the vertices $S$ and $T$ of any directed graph is a sequence of distinct edges of the form $(S_0,S_1),(S_1,S_2),\ldots ,(S_{n-1},S_n)$ where $S_0=S$ and $S_n=T$.

\begin{theorem}\label{tree}
    If $f\in \CaC\setminus\{0\}$, then $G(\mathfrak{B}(Fb=f))$ is a tree with root $(\CaC\setminus\CaB(f))\cup\{0\}$. Furthermore, the set of children of any $T\in \mathfrak{B}(Fb=f)$ is the set $\{T\cup \{x\} \mid x\in SG(T) \text{ and } x\prec \alpha(T)\}$. 
\end{theorem}
\begin{proof}
    Let $S\in \mathfrak{B}(Fb=f)$. We define recursively the following sequence:
\begin{eqnarray*}
S_0 &=& S, \\
S_{i+1} &=& \left\{ \begin{array}{lr}
    S_i\setminus\{\alpha(S_i)\} &  \text{ if } S_i \neq (\CaC\setminus \CaB(f))\cup\{0\}, \\
     S_i & \text{ otherwise.}
\end{array}\right.
\end{eqnarray*}
Since the set $ \CaB(f)$ is finite, the above sequence becomes stationary. Thus, any $S\in \mathfrak{B}(Fb=f)$ is connected by a path to 
$(\CaC\setminus \CaB(f))\cup\{0\}$, and the uniqueness of the path is deduced from the uniqueness of $\alpha(S)$. If $S$ is a child of $T$, then $T = S \setminus\{ \alpha(S) \}$, and therefore $S=T \cup \{ \alpha(S)\}$ is a $\CaC$-semigroup, which implies that $\alpha(S)\in SG(T)$ and $\alpha(S)\prec\alpha(T)$.
Conversely, if $x \in SG(T)$ and $x \prec \alpha(T)$, then $S = T \cup \{x\}$ which ensures $S$ is a $\CaC$-semigroup with $\alpha(S) = x$. Thus, $S$ is indeed a child of $T$.
\end{proof}

With the theoretical foundation established, we introduce Algorithm \ref{AlgcomputeCFr} for computing $\mathfrak{B}(Fb=f)$. Example \ref{ejemploArbol} illustrates it and 
shows the graph $G(\mathfrak{B}(Fb=f))$ obtained. 

\begin{algorithm}[H]
\caption{Computing the set $\mathfrak{B}(Fb=f)$.}\label{AlgcomputeCFr}
\KwIn{Let $f\in\CaC\setminus\{0\}$.}
\KwOut{The set $\mathfrak{B}(Fb=f)$.}

$A \leftarrow \{(\CaC\setminus\CaB(f))\cup\{0\}\}$\;
$X \leftarrow A$\;

\While {$A\ne\emptyset$}{
$Y \leftarrow \emptyset $\;

$B \leftarrow A$\;
\While {$B\ne\emptyset$}{
    $T \leftarrow \text{First}(B)$\;
    $C \leftarrow \{x\in SG(T)\mid x\prec \alpha(T)\}$\;
    \If{$C\ne\emptyset$}{
        $Y \leftarrow  Y \cup \{T\cup\{x\}\mid x\in C\}$\;\label{paso}
    }
    $B \leftarrow B\setminus \{T\}$\;
    }
    $X \leftarrow X \cup Y$\;
    $A \leftarrow Y$\;
    }
    \Return{$X$}
    
\end{algorithm}

\begin{example}\label{ejemploArbol}
Consider $\CaC= \N^2$. Fixed $f=(2,2)\in\CaC$ and the degree lexicographic order, the $\CaC$-semigroup $(\CaC\setminus\CaB(f))\cup\{0\}$ is shown in Figure \ref{ejemplofrakB}. The blue squares are the minimal generators, specifically,
    \begin{multline*}
        \msg\left((\CaC\setminus\CaB(f))\cup\{0\}\right)=\{(0, 3), (0, 4), (0, 5), (1, 3), (1, 4), (1, 5), (2, 3),\\ (2, 4), (3,0), (3, 1), (3, 2), (4, 0), (4, 1), (4, 2), (5, 0), (5, 1), (2, 5), 
(5, 2)\}.
    \end{multline*}%
The empty circles represent the set $\CaB(f)$  and the red circles are elements of $\CaC\setminus\CaB(f)$.  
\begin{figure}[ht]
    \centering
    \includegraphics[scale=.45]{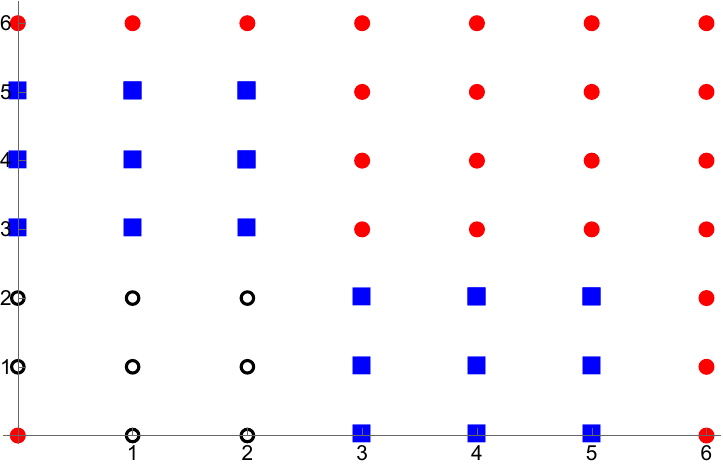}
    \caption{Example of $(\CaC\setminus\CaB(f))\cup\{0\}$.}
    \label{ejemplofrakB}
\end{figure}%
For this example, Figure \ref{arbolB_f} illustrates the $4$-level tree $G(\mathfrak{B}(Fb=f))$ defined in Theorem \ref{tree}. Its root is the $\CaC$-semigroup $(\CaC\setminus\CaB(f))\cup\{0\}$, and each node represents a $\CaC$-semigroup in $\mathfrak{B}(Fb=f)$. For example, the rightmost node $\{2,0\}$ in the last level is the $\CaC$-semigroup $\big(\CaC\setminus\CaB(f)\big)\cup \{(2,1),(1,2),(2,0),(0,0)\}$. In each level, some special gaps of each node are joined to obtain its children (step \ref{paso} in Algorithm \ref{AlgcomputeCFr}).
\begin{figure}[ht]
    \centering
    \includegraphics[scale=.44]{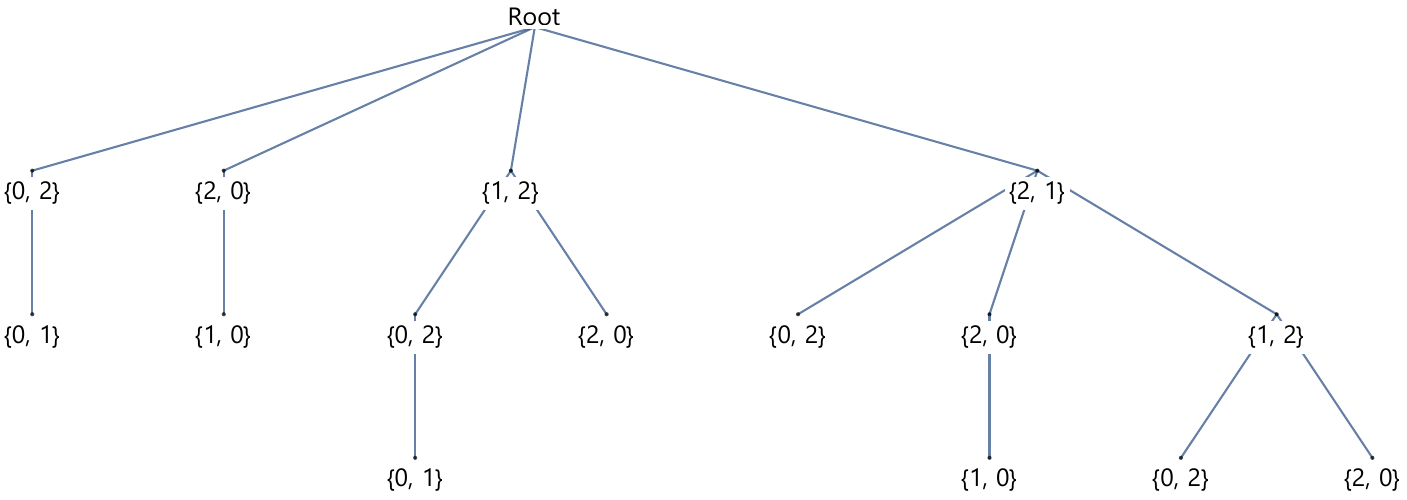}
    \caption{$G(\mathfrak{B}(Fb=(2,2)))$ with the degree lexicographic order.}
    \label{arbolB_f}
\end{figure}
Recall that the tree $G(\mathfrak{B}(Fb=f))$ depends on the fixed total order. For example, when the chosen total order is defined as $(a,b)\prec (c,d)$ if $2a+b<2c+d$, or, in case $2a+b=2c+d$, $a<c$, the obtained tree is shown in Figure \ref{arbolB_f2_2}.
\begin{figure}[ht]
    \centering
    \includegraphics[scale=.4]{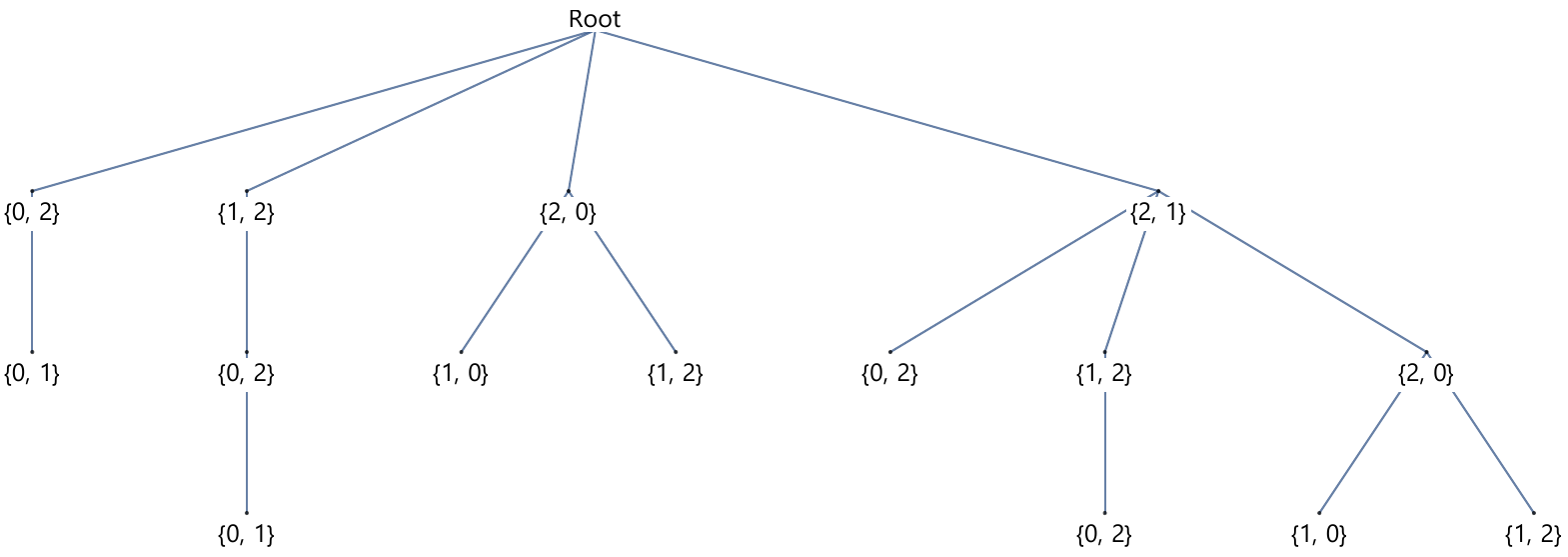}
    \caption{$G(\mathfrak{B}(Fb=(2,2)))$ with the total order $(a,b)\prec (c,d)$ iff $2a+b<2c+d$, or, in case $2a+b=2c+d$, $a<c$.}
    \label{arbolB_f2_2}
\end{figure}
   
\end{example}

We can provide a result equivalent to Theorem \ref{cacCFG} for the set $\mathfrak{B}(Fb=f, gen=g)$.

\begin{proposition}\label{cotaBsemig}
    For any $f\in \CaC\setminus\{0\}$ and any positive integer $g$, the set $\mathfrak{B}(Fb=f, gen=g)$ is non-empty if and only if
    \[\left\lceil \frac{\sharp \CaB(f)}{2} \right\rceil\leq g \leq\sharp \CaB(f)-1.
    \]
\end{proposition}
\begin{proof}
Consider that $\mathfrak{B}(Fb=f, gen=g)$ is non-empty, and let $S\in \mathfrak{B}(Fb=f, gen=g)$, then $\g(S)\leq \sharp \CaB(f)-1$. Analogously to the proof of Theorem \ref{cacCFG}, $\g (S)\geq  \big\lceil \frac{\sharp \CaB(f)}{2} \big\rceil$.

Conversely, since $\Delta(\CaC, f)\in \mathfrak{B}\left(Fb=f, gen=\big\lceil \frac{\sharp \CaB(f)}{2} \big\rceil\right)$, this set is connected with the root $(\CaC\setminus\CaB(f))\cup\{0\}$ in the tree $G(\mathfrak{B}(Fb=f))$ by removing an element in each level, and $(\CaC\setminus\CaB(f))\cup\{0\}\in \mathfrak{B}(Fb=f, gen=\sharp\CaB(f)-1)$, we can conclude that $\mathfrak{B}(Fb=f, gen=g)$ is non-empty for any $g\in\llbracket\big\lceil \frac{\sharp \CaB(f)}{2} \big\rceil,\sharp \CaB(f)-1\rrbracket$.
\end{proof}

Assuming $\CaC=\N^p$, we reformulated Corollary \ref{corocacCFG} as follows.

\begin{corollary}
    For any $f\in \N^p\setminus\{0\}$ and any positive integer $g$, the set $\mathfrak{B}(Fb=f, gen=g)$ is non-empty if and only if
    \[\left\lceil \frac{\prod_{i\in\llbracket p\rrbracket}(f_i+1)}{2} \right\rceil\leq g < \prod_{i\in\llbracket p\rrbracket}(f_i+1).
    \]
\end{corollary}

From the previous results, it is established that, for a fixed Frobenius element, the genus belongs to a bounded interval. Using Algorithm \ref{AlgcomputeCFr} as a basis, we determine the set $\mathfrak{B}(Fb = f, \text{gen} = g)$ by focusing specifically on the first $\sharp \CaB(f)-g$ steps. Continuing with the framework established in Section \ref{2}, we now address the computation of $\mathfrak{B}(\text{gen} = g)$. Rather than calculating $\mathfrak{B}(\text{gen} = g)$ directly, according to Proposition \ref{cotaBsemig}, we compute the set
$$\CaF_\mathfrak{B}(g)= \left\{f\in \CaC\mid  \left\lceil \frac{\sharp \CaB(f)}{2} \right\rceil\leq g \leq \sharp \CaB(f)-1\right \}.$$
For each $f \in \CaF_\mathfrak{B}(g)$, we determine $\mathfrak{B}(Fb = f, \text{gen} = g)$ as mentioned before. After repeating this procedure for each $f \in \CaF_\mathfrak{B}(g)$, we get the set $\mathfrak{B}(\text{gen} = g)$. The following example provides some computational results of this procedure.

\begin{example}
Let $\CaC=\N^2$, the set $\CaF_\mathfrak{B}(2)$ is $\{(0, 2), (0, 3), (1, 0), (1, 1), (2, 0), (3, 0)\}$, and $\mathfrak{B}(\text{gen} = 2)$ has also six elements, the $\CaB$-semigroups with gaps sets $\{(0,1),(0,2)\}$, $\{(0,1),(0,3)\}$, $\{(1,0),(1,1)\}$, $\{(0,1),(1,1)\}$, $\{(1,0),(2,0)\}$, and $\{(1,0),(3,0)\}$.

For genus five, the set $\CaF_\mathfrak{B}(5)$ obtained is again \eqref{conjuntoejemplogenero5} (as Example \ref{ejemplogenero5}), but only $58$ out of 210 elements in $\mathfrak{C}(\text{gen} = 5)$ are also $\CaB$-semigroups.

In general, the computation on $\CaC$-semigroups is very hard, and few examples can be constructed. For $\CaC=\N^2$, Table \ref{tab2} collects the number of $\CaB$-semigroups for some genus $g$.

\begin{table}[h]
    \centering
    \begin{tabular}{|c|c|c|c|c|c|c|c|c|c|}
    \hline
      $g$  &  1 &  2 & 3 & 4 & 5 & 6 & 7 & 8 & 9  \\ \hline
        $\sharp \mathfrak{B}(\text{gen} = g)$ & 2 & 6 & 15 & 30 & 58 & 137 & 240 & 457 & 900  \\\hline
    \end{tabular}
    \caption{For $\CaC=\N^2$, $\sharp \mathfrak{B}(\text{gen} = g)$ for some $g\in\N$.}
    \label{tab2}
\end{table}
\end{example}

Computational results seem to suggest that the following conjecture could be true.

\begin{conjecture}
    For any integer cone $\CaC\subseteq \N^p$, and any non-zero $g\in\N$,
    $$\sharp \mathfrak{B}(\text{gen} = g)<\sharp \mathfrak{B}(\text{gen} = g+1).$$
\end{conjecture}

\section{A partition of $\mathfrak{C}(Fb=f)$}\label{C-smgpFb}

Let us start by introducing some notations. Again, an integer cone $\CaC\subseteq \N^p$ and a total order $\preceq$ on $\N^p$ are fixed. 
Given $f \in \CaC \setminus \{0\}$ and $S,T\in\mathfrak{C}(Fb=f)$, we define an equivalence relation $\sim$ such that $S \sim T$ if and only if $\mathcal{O}(S) = \mathcal{O}(T)$, recall that $\mathcal{O}(S) = S \cup (\CaC \setminus \CaB(f))$. Equivalently, $S \sim T$ if and only if $S \cap \CaB(f) = T \cap \CaB(f)$. For $S \in \mathfrak{C}(Fb = f)$, the equivalence class of $S$ modulo $\sim$, called its $\sim$-class, is defined as $[S] = \{T \in \mathfrak{C}(Fb = f) \mid S \sim T\}$.
The collection of all $\sim$-classes forms a partition of $\mathfrak{C}(Fb = f)$, denoted by $\mathfrak{C}(Fb = f) /_\sim$. 

This section aims to compute $\mathfrak{C}(Fb=f)$. To this end, we study the partition $\mathfrak{C}(Fb = f) /_\sim$ as a tool for the desired computation. We show how the set $[S]$ can be arranged in a tree for any $S \in \mathfrak{C}(Fb = f)$. 

The following proposition establishes the relationship between the partition $\mathfrak{C}(Fb = f) /_\sim$ and the set $\mathfrak{B}(Fb = f)$, proving that the set $\mathfrak{C}(Fb = f) /_\sim$ can be determined from the $\CaB$-semigroups with Frobenius element $f$.

\begin{proposition}\label{PropUnion}
    If $f \in \CaC \setminus \{0\}$, then $\mathfrak{C}(Fb = f) /_\sim=\{[R]\mid R \in  \mathfrak{B}(Fb = f)\}$. Moreover, if $R_1,R_2\in  \mathfrak{B}(Fb = f)$ such that $R_1\ne R_2$, then $[R_1]\cap[R_2]=\emptyset$.
\end{proposition}

\begin{proof}    
    To prove the first statement, consider  $S\in \mathfrak{C}(Fb = f)$, by definition,   $[S] = [R]$, where $R=\mathcal{O}(S)\in \mathfrak{B}(Fb = f)$. For the second statement, assume $R_1,R_2\in  \mathfrak{B}(Fb = f)$ such that $R_1\ne R_2$, this implies that there is at least one gap $x$ of $R_i$ such that $x\in R_j$ with $i,j\in\{1,2\}$ and $i\ne j$. Thus $R_1 \not\sim R_2$, hence $[R_1] \cap [R_2] = \emptyset$.    
\end{proof}

Let $S$ be a $\CaC$-semigroup with Frobenius element $\Fb(S)=f$, we define $\lambda(S)$ as the maximum element in $\left(\CaC\setminus \CaB(f)\right)\cap \CaH(S)$ with respect to the fixed total order on $\N^p$.  If $S=(\CaC\setminus \CaB(f))\cup\{0\}$, by convention $\lambda(S)=0$. We deduce the following result from the maximality of $\lambda(S)$. Let $R \in \mathfrak{B}(Fb = f)$.
 
\begin{lemma}
    If $S\in [R]$, then $S\cup\{\lambda(S)\}\in [R]$.
\end{lemma}
 
We define the directed graph $G([R])$ whose vertex set is the $\sim$-class $[R]$, and  $(S, T) \in [R]^2 $ is directed  edge if and only if $T=S\cup\{\lambda(S)\}$. In particular, as mentioned before, if $(S, T)$ is a directed edge, $S$ is usually known as a child of $T$. 

\begin{theorem}
    The graph $G([R])$ is a tree with root $\mathcal{O}(R)$. Furthermore, the set of children of any $T\in [R]$ is the set 
    \[\{T \setminus\{x\}\mid x\in msg(T) \text{ and } 0\neq \lambda(T)\prec x\prec f
    \}.\]
\end{theorem}

\begin{proof}
     Let $S\in [R]$ such that $S\ne \mathcal{O}(R)$. We construct the sequence $\{ S_i\}_{i\in \N}\subseteq[R]$ defined by 
     \begin{eqnarray*}
S_0 &=& S \\
S_{i+1} &=& \left\{ \begin{array}{lr}
    S_i\cup\{\lambda(S_i)\} &  \text{ if } S_i \ne \mathcal{O}(R), \\
     S_i & \text{ otherwise.}
\end{array}\right.
\end{eqnarray*}
Since $\lambda(S_i)\in \CaH(S_i)$, and each time we add $\lambda(S_i)$ to $S_i$, the set of remaining possible $\lambda$-values decreases, the above sequence becomes stationary, and thus the sequence $\{S_i\}_{i\in \N}$ defines a path from $S$ to $\mathcal{O}(R)$. By the maximality property of $\lambda(S)$, we deduce the uniqueness of the path from $S$ to $\mathcal{O}(R)$. Consider $S=T\setminus\{x\}$ for some $x\in msg(T)$ such that 
$0\neq \lambda(T)\prec x\prec f$. Therefore, $T=S\cup \{\lambda(S)\}$, which proves that $S$ is a child of $T$. Now, let $S$ be a child of $T$, then $T=S\cup\{\lambda(S)\}$, which implies  $\lambda(S)$ is a minimal generator of $T$.  Finally, note that $\lambda\left(T\setminus\{\lambda(S)\}\right)=\lambda(S)$, which completes the proof. 
\end{proof}

The above results allow us to present Algorithm \ref{AlgCompute[R]} for computing the $\sim$-class $[S]$ for any $S\in \mathfrak{C}(Fb=f)$.

\begin{algorithm}[H]
\caption{Computing the $\sim$-class $[S]$.}\label{AlgCompute[R]}
\KwIn{A $\CaC$-semigroup $S$ with Frobenius element $f \in \CaC \setminus \{0\}$.}
\KwOut{The $\sim$-class $[S]$.}
$A \leftarrow \{\mathcal{O}(S)\}$\; 
$B \leftarrow A$\;  
\While {$A \ne \emptyset$}{
    $T \leftarrow \text{First}(A)$\;  
    $C \leftarrow \{x\in msg(T)\mid 0\neq \lambda(T)\prec x\prec f\}$\;

\label{otropaso}    $B \leftarrow B \cup \{T\setminus\{x\}\mid x\in C\}$\; 
    $A \leftarrow  (A\setminus \{T\}) \cup \{T\setminus\{x\}\mid x\in C\}$\;
    
}
\Return{$B$}  
\end{algorithm}

\begin{example}
As in Example \ref{ejemploArbol}, consider $\CaC= \N^2$, $f=(2,2)\in\CaC$ and the degree lexicographic order. Let $S\in \mathfrak{C}(Fb = f)$ be the semigroup minimally generated by 
\begin{multline*}
    \{(0, 4), (0, 5), (1, 2), (1, 4), (2, 1), (2, 3), (3, 0), (3, 1), (3,
2), (4, 0),\\ (4, 1), (5, 0), (0, 6), (0, 7), (1, 5)\},
\end{multline*}
with $\CaH(S)=\{(0, 1), (0, 2), (0, 3), (1, 0), (1, 1), (1, 3), (2, 0), (2, 2)\}$. The $\CaB$-semigroup $\mathcal{O}(S)$ is represented in Figure \ref{O_de_S}, where the empty circles are the gaps of $\mathcal{O}(S)$, the blue squares are the minimal generators of $\mathcal{O}(S)$, and the red circles are elements of $\mathcal{O}(S)$.
\begin{figure}[ht]
    \centering
    \includegraphics[scale=.5]{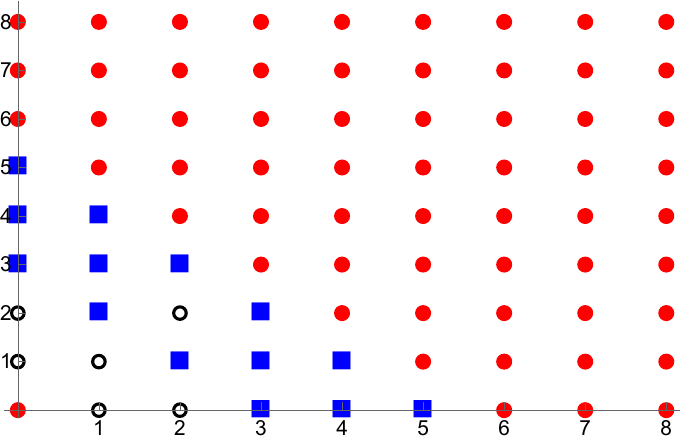}
    \caption{Example of $\mathcal{O}(S)$.}
    \label{O_de_S}
\end{figure}

From $\mathcal{O}(S)$, Figure \ref{arbolO_de_S} shows the tree $G([S])$ containing all the elements in the $\sim$-class of $S$. Its root is $\mathcal{O}(S)$, and each node represents a $\CaC$-semigroup in $\mathfrak{C}(Fb=f)$.  To ensure more clarity in the figure, each tree vertex is labelled with the element removed to reach its parent node. For example, the leftmost node  labelled $\{1,3\}$ in the last level is the $\CaC$-semigroup $\mathcal{O}(S)\setminus \{(0,3),(3,0),(0,4),(1,3)\}$. In each level, some minimal generators of each node are removed to obtain its children (step \ref{otropaso} in Algorithm \ref{AlgCompute[R]}).
\begin{figure}[ht]
    \centering
    \includegraphics[scale=.43]{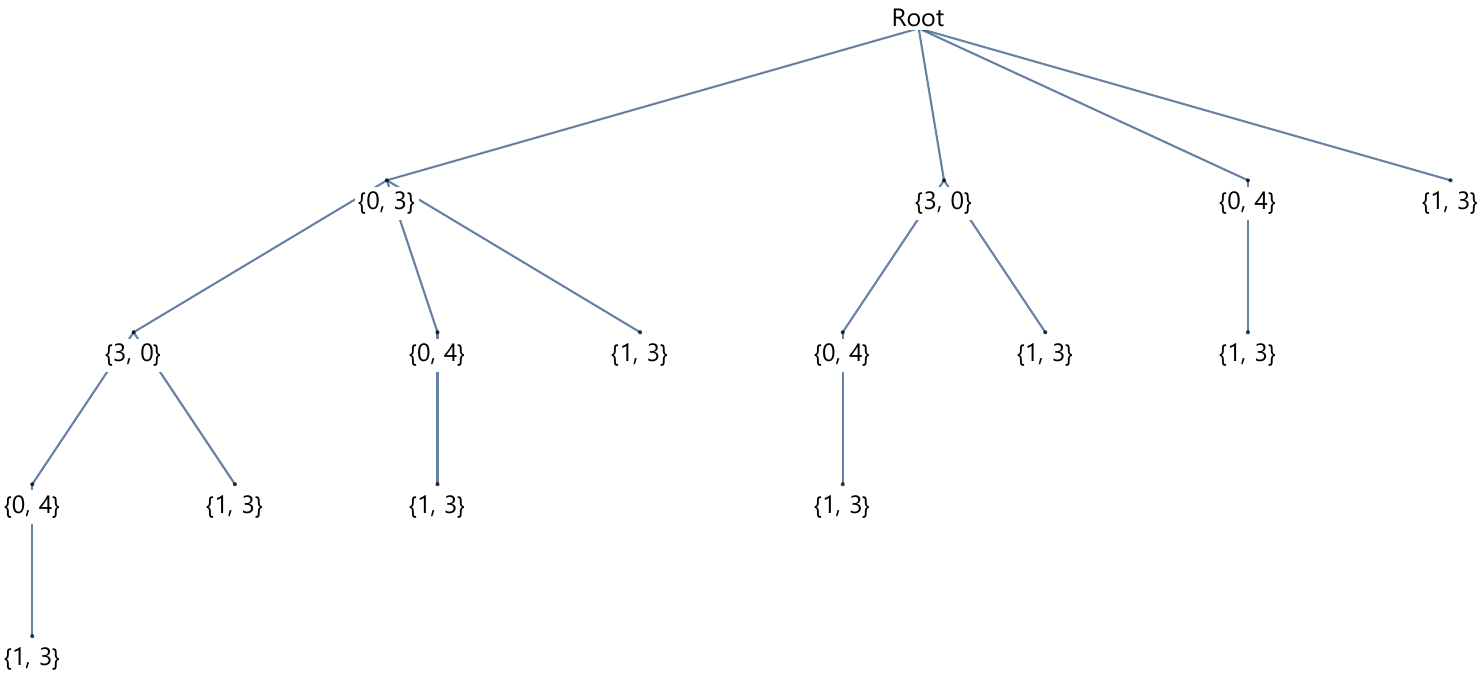}
    \caption{$G([S])$ with the degree lexicographic order.}
    \label{arbolO_de_S}
\end{figure}

\end{example}

We have already developed all the necessary background for computing $\mathfrak{C}(Fb = f)$: Proposition \ref{PropUnion} states that $\mathfrak{C}(Fb = f)$ can be expressed as the union $\bigcup_{S \in \mathfrak{B}(Fb=f)} [S]$, and Algorithm \ref{AlgCompute[R]} computes $[S]$. Thus, by combining these two ideas, we achieve a procedure for computing all $\mathfrak{C}(Fb = f)$.

\begin{example}
Fix the degree lexicographic order and let $\CaC=\N^2$. Figure \ref{arbolB_f} in Example \ref{ejemploArbol} contains all the elements in $\mathfrak{B}(Fb=(2,2))$, that is, all the elements needed to construct the union $\bigcup_{S \in \mathfrak{B}(Fb=(2,2))} [S]$. After to compute $[S]$ (Algorithm \ref{AlgCompute[R]}) for every $S \in \mathfrak{B}(Fb=(2,2))$, we concluded that there exists $202$ $\N^2$-semigroups with Frobenius element $(2,2)$ respect to the degree lexicographic order. However, if we consider the total order used to construct Figure \ref{arbolB_f2_2}, then there are $500$ elements in $\mathfrak{C}(Fb = f)$. Recall that all these sets are highly dependent on the total order considered.
\end{example}

\section{Ordinary $\N^2$-semigroups}\label{Ordinary}

Fixed a total order $\preceq$, we say that a $\CaC$-semigroup $S$ is ordinary if $S=\{0\}\cup\{x\in \CaC\mid x\succeq c\}$ for some $c\in \CaC$, and it is denoted by $S_c$ which depends on the choice of the order. Equivalently, $S$ is ordinary if the conductor of $S$ equals the multiplicity of $S$. Note that an ordinary $\CaC$-semigroup is a $\CaC$-semigroup containing all the $\CaC$-semigroups determined by the given total order and the multiplicity.

Our goal in this section is to prove that any ordinary $\N^2$-semigroup ordered by the graded lexicographic order satisfies the Generalized Wilf Conjecture and the Extended Wilf Conjecture. We study the minimal generating set of any ordinary $\CaC$-semigroup to achieve this. First, we provide a lower bound for the embedding dimension of any ordinary $\CaC$-semigroup.

\begin{proposition}
    Let $S$ be an ordinary $\CaC$-semigroup. Then $\g(S)<\e(S)$.
\end{proposition}
\begin{proof}
     We fix a total order $\preceq$ and assume that $S$ is an ordinary $\CaC$-semigroup. Trivially, if $S=\CaC$, then $\e(S)>0$. We proceed by induction on $\g(S)$. Suppose that such that $\g(S) < \e(S)$, and we show that $\g(S \setminus \{m\}) < \e(S \setminus \{m\})$ where $m=\m(S)$. Note that $\g(S \setminus \{m\})=\g(S)+1$. Consider $r = \min_\preceq (S \setminus \{0, m\}),$ and $t = \min_\preceq (S \setminus \{0, m, r\})$ the second and third minimum elements of $S \setminus \{0\}$, respectively. Let $z \in \{2m, m + r, m + t\}$, if $z$ is not a minimal generator of $S \setminus \{m\}$, then $z = a + b$, where $a$ and $b$ are two non-zero elements of $S \setminus \{m\}$. Without loss of generality, we deduce $a \prec m$, which leads to a contradiction, since $m$ is the multiplicity of $S$. So, $2m, m + r, m + t$ are minimal generators of $S\setminus\{m\}$. From \cite[Lemma 3]{Wilf} follows $\msg(S)\setminus\{m\}\subset \msg(S\setminus\{m\})$, hence $\e(S\setminus \{m\})\geq \e(S)-1+3$. By applying the induction hypothesis, $\g(S \setminus \{m\})=\g (S)+1 < \e(S)+1< \e(S\setminus \{m\})$, which completes the proof.
\end{proof}

To emphasize the relevance of the ordinary $\CaC$-semigroups, we introduce two families of $\CaC$-semigroups: arf semigroups and saturated semigroups, and we prove that any ordinary $\CaC$-semigroup belongs to both families. Recall that for any positive integer $b$, 
we denote the set $\llbracket b \rrbracket=\{0,1, 2, \ldots, b\}$.

 We say that an affine semigroup $S$ is an arf (affine) semigroup if, for any $x,y,z\in S$ with $x\geq_\CaC y \geq_\CaC  z$, then $x+y-z\in S$. We say that $S$ is a saturated (affine) semigroup if $s,s_1,\ldots, s_r\in S$ are such that $s_i\leq_\CaC s$ for all $i\in\llbracket r\rrbracket\setminus\{0\}$ and $z_1,\ldots, z_r\in\Z$ are such that $z_1s_1+\cdots+z_rs_r\in \CaC$  then, $s+z_1s_1+\cdots+z_rs_r\in S$. In the case of numerical semigroups, both classes of semigroups have been studied in the literature (see, for example, \cite{arfnumerical} and \cite{saturated}).

The following lemma generalizes Proposition 3.31 in \cite{libroRosales} from numerical semigroups to affine semigroups.

\begin{lemma}\label{satu-arf}
    Every saturated semigroup is an arf semigroup. 
\end{lemma}
\begin{proof}
    Let $S$ be a saturated affine semigroup. Take $x,y,z\in S$ such that $x\geq_\CaC y \geq_\CaC  z$, implying $y-z$ belongs to $\CaC$. By definition of being saturated, we have that $x+y-z\in S$.
\end{proof}

\begin{proposition}
    Every ordinary $\CaC$-semigroup is a saturated semigroup.
\end{proposition}
\begin{proof}
Let $S_c$ an ordinary  $\CaC$-semigroup. Suppose $s,s_1,\ldots, s_r\in S_c$ such that $s_i\leq_\CaC s$ for all $i\in \llbracket r \rrbracket\setminus\{0\}$ and $z_1,\ldots, z_r\in\Z$  such that $t=z_1s_1+\cdots+z_rs_r\in \CaC$.  By hypothesis $s=s_1+s_2+\cdots +s_r+\lambda$, for some $\lambda\in \CaC$. Since $s+t\geq s$ and $S_c$ is an ordinary, then $s+t\in S_c$.
\end{proof}

\begin{corollary}
    Every ordinary $\CaC$-semigroup is an arf semigroup.
\end{corollary}

The following definitions will be needed throughout the remainder of the work. Let $w = (w_1, \ldots, w_p) \in \mathbb{R}_\geq^p$ be a vector, and consider the map $\pi_w : \mathbb{N}^p \to \mathbb{N}$ defined via $\pi_w(x) = w \cdot x$, where $\cdot$ denotes the inner product. For any $x,y \in \mathbb{N}^p$, we define $x \preceq_w y$ if and only if $\pi_w(x) \leq \pi_w(y)$. We refer to $\preceq_w$ as the weight order determined by $w$.

To determine a weight order in general, we choose a primary weight vector $w \in \mathbb{R}_\geq^p$. A secondary weight vector $u \in \mathbb{R}_\geq^p$ is employed to break ties. If ties persist (i.e., when $\pi_w(x) = \pi_w(y)$ and $\pi_u(x) = \pi_u(y)$), a third weight vector is introduced, and so on. Thus, every monomial order $\preceq$ can be obtained through this finite process of applying weight vectors. From now on, we are interested in the first weight vector. Hence, we use $\pi_{\preceq}$ instead of $\pi_w$. Graphically, it can be interpreted as the existence of a hyperplane that separates the space into two regions, one containing $x$ and the other containing $y$. For a detailed treatment of monomial orders and their relation to weight orders, consult \cite{Cox} and \cite{Robbiano}. We assume that the vector $w$ defining the fixed order has non-zero entries.

Let $S$ be a $\CaC$-semigroup with $t$ external rays, and minimally generated by $\msg(S)=E\sqcup A$ with $E=\cup_{i\in \llbracket t \rrbracket\setminus\{0\}} \mult_i (S)=\{m_1,\ldots,m_t\}$, and $A=\{m_{t+1},\ldots ,m_r\}$.  The next result states that the minimal system of generators of an ordinary $\CaC$-semigroup is bounded.

\begin{lemma}\label{upperboundMSG}
    Let $S_c$ be an ordinary  $\CaC$-semigroup and $x\in \msg(S_c)$. Then, $\pi_\preceq(x)\leq\pi_\preceq(m)$, where $m=\sum_{i\in \llbracket t \rrbracket\setminus\{0\}}\mult_i(S_c)$. 
\end{lemma}
\begin{proof}
    Let $x\in S_c$ such that $\pi_\preceq(x)>\pi_\preceq(m)$ we have to show that $x\notin \msg(S_c)$. We assume that $x-m_i\notin S_c$ for any $i\in \llbracket t \rrbracket \setminus\{0\}$, otherwise it is easily checked that $x$ is not a minimal generator. We distinguish the following cases:
    \begin{itemize}
        \item If there exists $j\in  \llbracket t \rrbracket \setminus\{0\}$ such that $x-m_j\in \CaC$, then $x=m_j+h$ for some $h\in \CaH(S_c)$. Given that $\pi_\preceq$ is a linear map, we have $\pi_\preceq(x)=\pi_\preceq(m_j)+\pi_\preceq(h)>\sum_{i\in  \llbracket t \rrbracket \setminus\{0\}}\pi_\preceq(m_i)$. This implies that $\pi_\preceq(h)>\sum_{i\in  \llbracket t \rrbracket\setminus\{j\}}\pi_\preceq(m_i)>\pi_\preceq(c)$. Since $\pi_\preceq$ is an increasing map, $h\succ c$, and thus, $h\in S_c$, which it is not possible.

        \item If $x-m_i\notin \CaC$ for any $i\in \llbracket t \rrbracket \setminus\{0\}$, since $x$ belongs to $S_c\subseteq\CaC$, there exist some rational numbers $0\le \lambda_1,\ldots, \lambda_t<1$ such that $x=\sum_{i\in  \llbracket t \rrbracket \setminus\{0\}}\lambda_im_i$. So, $\pi_\preceq(x)=\sum_{i\in  \llbracket t \rrbracket \setminus\{0\}}\lambda_i\pi_\preceq(m_i)<\sum_{i\in  \llbracket t \rrbracket}\pi_\preceq(m_i)=\pi_\preceq(m)$, in contradiction with the hypothesis.
     \end{itemize}
\end{proof}

Building on this, the following lemma proves a more robust result about the relationship between some elements in $\CaC$ and the minimal system of generators of $S_c$. From now on, we assume that the conductor $c$ is non-zero.

\begin{lemma}\label{lowerboundMSG}
    Let $S_c$ be an ordinary $\CaC$-semigroup and $x\in S_c$ such that  $\pi_\preceq(x)< 2\pi_\preceq(c)$, then $x\in \msg(S_c)$.
\end{lemma}

\begin{proof}
    Let $x\in S_c$  such that  $\pi_\preceq(x)< 2\pi_\preceq(c)$, and suppose that $x=s_1+s_2$ for some $s_1,s_2\in S_c$. Since $S_c$ is an ordinary $\CaC$-semigroup, $\pi_\preceq(s_1), \pi_\preceq(s_2)\geq\pi_\preceq(c)$. Consequently $\pi_\preceq(x)=\pi_\preceq(s_1)+ \pi_\preceq(s_2)\geq2\pi_\preceq(c)$, contradicting the initial assumption.
\end{proof}

In particular, if $S_c$ is an ordinary $\N^2$-semigroup, and $\preceq_{glex}$ corresponds to the graded lexicographic order, the set $\msg(S_c)$ can be explicitly determined.
Note that in this context, the map $\pi_{\preceq_{glex}}:\N^2\rightarrow\N$ is defined via $\pi_{\preceq_{glex}}(x_1,x_2)=x_1+x_2$. From now on,  we use the symbol $\pi$ instead of $\pi_{\preceq_{glex}}$ for short.

\begin{proposition}\label{propOR0}
    Let $S_c$ be an ordinary $\N^2$-semigroup with conductor $c = (0, c_2)$, ordered by $\preceq_{glex}$. Then, $S_c$ is minimally generated by
    \[ \left\{ x \in \N^2 \mid \pi(c) \leq \pi(x) \leq 2 \pi(c) - 1 \right\}.\]
    Therefore,  $\e(S_c)=\dfrac{c_2(3c_2+1\big)}{2}$.
\end{proposition}
\begin{proof}
    Since $\pi(\mult_i(S_c)) = \pi(c)$ for each $i = 1, 2$, and by applying Lemmas \ref{upperboundMSG} and \ref{lowerboundMSG}, it suffices to show that any element $x = (x_1, x_2) \in \N^2$ satisfying $\pi(x) \geq 2 \pi(c)$ cannot be a generator of $S_c$. If $x_1,x_2<c_2$, then $\pi(x)< 2\pi(c)$ which it is impossible by hypothesis. Thus, without loss of generality, we assume $x_1\geq c_2$. It follows that $x-(c_2,0)\in \N^2$, and since  $\pi(x-(c_2,0))\geq\pi(c)$, it can be deduced that $x-(c_2,0)\in S_c$. Therefore, $x\notin \msg(S_c)$.

    The embedding dimension of $S_c$ is determined by the number of solutions $(x_1,x_2)\in \N^2$ of $c_2\leq x_1+x_2 \leq 2c_2-1$. Equivalently, it is the number of natural solutions of the equation of $x_1+x_2=k$, where $k\in \{c_2, c_2+1, \ldots ,2c_2-1\}$. Note that the above expression involves $2c_2$ equations. Fixed $k$, it is straightforward to deduce that $x_1\in \{0,1,\ldots, k\}$ and consequently, $x_2=k-x_1$ then, there exist $k+1$ solutions. Thus, 
    \begin{eqnarray*}
        \e(S_c)&=& \sum_{k=c_2}^{2c_2-1} (k+1)=\sum_{i=1}^{c_2}(c_2+i)= \dfrac{c_2(3c_2+1\big)}{2}.
    \end{eqnarray*}
\end{proof}

The above result corresponds to a particular case discussed within the framework of T-stripe generalized numerical semigroups (see \cite[Proposition 3.3]{Cisto2022}). In cases where $c_1 \neq 0$, the determination of $\msg(S_c)$ has not been addressed in the literature, as such semigroups do not fall under the classification of T-stripe $\N^2$-semigroups. Therefore, we propose the following proposition. For any total order $\preceq$, let $a \preceq b$ be two elements in $\CaC$. We denote the intervals of elements in $\CaC$ between $a$ and $b$ under the order $\preceq$ as follows: $[a, b]_\preceq, [a, b)_\preceq, (a, b]_\preceq$ and $(a, b)_\preceq$ for the closed interval, left-open interval, right-open interval and open interval, respectively.

\begin{proposition}\label{propOr2}
    Let $S_c$ be an ordinary $\N^2$-semigroup with conductor $c = (c_1, c_2)$ such that $c_1\ne 0$, ordered by $\preceq_{glex}$. Then, $S_c$ is minimally generated by
    \begin{equation}\label{msg}
        \Big[c,\;2c\Big)_{\preceq_{glex}}  \sqcup  
        \Big[\big(0,2\pi(c)+1\big),\; c+\big(0,\pi(c)+1\big)\Big)_{\preceq_{glex}}
    \end{equation}
    and $\e(S_c)=\dfrac{3\pi^2(c) +\pi(c)+4c_1}{2}$.
\end{proposition}
\begin{proof}
    From direct application of Lemma \ref{upperboundMSG}, the set \eqref{msg} is a system of generators of $S_c$, since $\mult_1(S_c)+ \mult_2(S_c)= 2\pi(c)+1$. Furthermore, by Lemma \ref{lowerboundMSG} it suffices to analyse those $x=(x_1,x_2)$ such that $2\pi(c)\leq\pi(x)\leq 2\pi(c)+1$. We distinguish two cases: $x\not> c$ and $x>c$.

    If $x\not>  c$, where $>$ denotes the component-wise order, then  $x_2<c_2$ or $x_1<c_1$. Suppose that $x_2<c_2$, by hypothesis $2c_1+2c_2\leq x_1+x_2 <x_1+c_2$, hence $\pi(c)+c_1<x_1$. Consider $y=x-(\pi(c),0)\in \N^2$. Since $\pi(y)\geq 2\pi(c)-\pi(c)=\pi(c)$, we conclude that $x\notin \msg(S_c)$. Now, assume that $x_1<c_1$, and suppose that $x=s+t$ for some $s,t\in S_c$. Since $x_1<c_1$ then $s_1,t_1<c_1$. By the linearity of $\pi$, we obtain that $2\pi(c)\leq \pi(x)\leq 2\pi(c)+1$. Without loss of generality, suppose that $\pi(s)=\pi(c)$; otherwise, $\pi(t)=\pi(c)$ and the argument is analogous. By definition of $S_c$, we have that $s\notin S_c$, which is a contradiction. Therefore, we conclude that the set \eqref{msg} is contained in $\msg(S_c)$.

    If $x>c$, then there exist $\lambda_1, \lambda_2\in \N$ such that  $x=c+(\lambda_1, \lambda_2)$ and $\lambda_1+ \lambda_2\ne 0$. We distinguish two cases. When $\pi(x)=\pi(c)$, necessarily, $\pi((\lambda_1, \lambda_2))=\pi(c)$. If $\lambda_1>c_1$ then $(\lambda_1, \lambda_2)\in S_c$, and thus $x\notin \msg(S_c)$. If $\lambda_1<c_1$ then $x_1\leq2c_1$. Assuming $x=s+t$ for some $s,t\in S_c$, and repeating the mentioned argument, we deduce that $\pi(s)=\pi(t)=\pi(c)$. By definition of $S_c$, it follows that $s_1,t_1>c>c_1$, therefore $x=s_1+t_1\geq 2c_1$, which is a contradiction. Hence, $x\in \msg(S_c)$. So, $\Big[c+\big(0, \pi(c)\big) \;, 2c\Big)_{\preceq_{glex}}\subset \msg(S_c)$. Finally, when $\pi(x)=2\pi(c)+1$, then $\pi(c)+1=\pi((\lambda_1, \lambda_2))$, and thus $\pi((\lambda_1, \lambda_2)) > \pi(c)$, it follows that
    $(\lambda_1, \lambda_2)\in S_c$. Hence, $x\notin \msg(S_c)$. Which finishes the first part of the proof.

    The computation of $\e(S_c)$ relies on the cardinality of the following disjoint sets since the set \eqref{msg} can be rewritten depending on $\pi$ as
    \begin{eqnarray}
        msg(S_c) &=& \{ x \in \mathbb{N}^2 \mid \pi(x) = \pi(c) \text{ and } x_1 \geq c_1 \} \label{A_2} \\ 
        &\sqcup& \{ x \in \mathbb{N}^2 \mid \pi(c) + 1 \leq \pi(x) \leq 2\pi(c) - 1 \}  \label{A_3} \\
        &\sqcup& \{ x \in \mathbb{N}^2 \mid \pi(x) = 2\pi(c) \text{ and } x_1 < 2c_1 \} \label{A_4} \\
        &\sqcup& \{ x \in \mathbb{N}^2 \mid \pi(x) = 2\pi(c) + 1 \text{ and } x_1 < c_1 \} \label{A_5}.
    \end{eqnarray}
   The cardinality of the set \eqref{A_2} is determined by the number of solutions $(x_1,x_2)\in \N^2$ of $x_1+x_2=\pi(c)$ such that $x_1\geq c_1$, we deduce that $ x_1 \in \{ c_1, c_1 + 1, \ldots, \pi(c) \} $ and consequently, $ x_2 = \pi(c) - x_1 $. Therefore,  there are $c_2+1$ different solutions. For the cardinality of the set \eqref{A_3}, by arguing as in the proof of Proposition \ref{propOR0}, is determined by
    \[
    \sum_{k=\pi(c)+1}^{2\pi(c)-1}(k+1)=\frac{(\pi(c) - 1)(3\pi(c) + 2)}{2}.
    \]
    Regarding sets \eqref{A_4} and \eqref{A_5}, their cardinalities correspond to the number of natural solutions $(x_1,x_2)\in \N^2$ of  $x_1+x_2=2\pi(c)$ and  $x_1+x_2=2\pi(c)+1$, respectively, satisfying  the conditions $x_1<2c_1$, and $x_1<c_1$ respectively. Therefore, the cardinality of \eqref{A_4} is $2c_1$ and for the set \eqref{A_5} is $c_1$.
    Consequently, adding the previous cardinalities, we have $\e(S_c)= \dfrac{3\pi^2(c) +\pi(c)+4c_1}{2}$.
\end{proof}

\begin{remark}\label{e(S_c)coin}
    From previous results, we observe that the minimal generating set of an ordinary $\N^2$-semigroup (ordered by $\preceq_{glex}$) depends on the first coordinate of its conductor, and that $\dfrac{3\pi^2(c) +\pi(c)+4c_1}{2}$ equals $\dfrac{c_2(3c_2+1\big)}{2}$ for $c_1=0$. This fact simplifies the proof of the Generalized Wilf Conjecture and Extended Wilf Conjecture hold.
\end{remark}

\begin{conjecture}\cite[Conjecture 2.8]{CistoWilf}\label{WilfCisto}
    Let $S$ be a $\N^p$-semigroup. The Generalized Wilf Conjecture is 
    \begin{equation}
    \label{InewilfCisto}
        \nu(S) e(S)\geq p\, \gamma(S),
    \end{equation}
    where $\nu(S)=\sharp\{x\in S \mid x\leq_{\N^p} h \text{ for some } h\in \CaH(S)\}$, and $\gamma(S)=\sharp\{x\in \N^p \mid x\leq_{\N^p} h \text{ for some } h\in \CaH(S)\}$.
\end{conjecture}

\begin{conjecture}\cite[Conjecture 14]{Wilf}\label{WilfAlb}
    Let $S$ be a $\CaC$-semigroup. The Extended Wilf Conjecture is
    \begin{equation}\label{Inewilf}
        \n(S) \e(S) \geq \CaN (\Fb(S))+1.
    \end{equation}
\end{conjecture}

The following proposition establishes the relation between the Conjectures \ref{WilfCisto} and \ref{WilfAlb}.

\begin{proposition}\cite[Proposition 6.3]{CistoWilf}\label{propAUX}
    If $S\subseteq \N^p$ is a generalized numerical semigroup that satisfies the Generalized Wilf Conjecture, then $S$ satisfies the Extended Wilf Conjecture
\end{proposition}
As a consequence of Propositions \ref{propOR0} and \ref{propOr2}, we obtain the desired result.

\begin{corollary}\label{coroORWilf}
    Every ordinary $\N^2$-semigroup, ordered by $\preceq_{glex}$, satisfies the  Generalized Wilf Conjecture and the Extended Wilf Conjecture.
\end{corollary}
\begin{proof}
    Applying Proposition \ref{propAUX} suffices to prove inequality \eqref{InewilfCisto}. If $S=\N^2$, by convention $\gamma(S)=0$, and the inequality \eqref{InewilfCisto} is trivial.  Let $S_c$ be an ordinary $\N^2$-semigroup with non-null conductor $c = (c_1, c_2)$, ordered by $\preceq_{glex}$. Clearly $n(S_c)=1$, and $\gamma(S)=c_2$ if $c_1=0$, thus
    \[
    \frac{e(S)}{2}=\dfrac{c_2(3c_2+1)}{4}\geq \gamma(S).
    \]
    If $c_1\ne 0$ then $\gamma(S)=c_1(c_2+2)$. Therefore, the inequality \eqref{InewilfCisto} is equivalent to
    \[
    \frac{e(S)}{2}=\dfrac{3c_1^2+6c_1c_2+3c_2^2+5c_1+c_2}{4}\geq c_1(c_2+2), 
    \]
    which is true for all natural numbers $c_1$ and  $c_2$.
\end{proof}

To illustrate the results discussed, we provide the following example.

\begin{example}
Let $S_c$ be an ordinary $\N^2$-semigroup with conductor $c = (7,3)$, ordered by $\preceq_{glex}$. By Proposition \ref{propOr2}, we obtain that $S_c$ is minimally generated by 
\[
\Big[(7,3), (14,6)\Big)_{\preceq_{glex}} \sqcup \Big[(0, 21), (7, 14)\Big)_{\preceq_{glex}}
\]
and its embedding dimension is $\e(S_c) = 2 \cdot 7 + \frac{10 \cdot 31}{2} = 14 + 155 = 169$.
Figure \ref{Ej-ordinario} gives a graphical representation of $S_c$. The empty circles are the gaps of $S_c$, the blue squares are the minimal generators of $S_c$, and the red circles are elements of $S_c$.
\begin{figure}[ht]
    \centering
    \includegraphics[scale=.3]{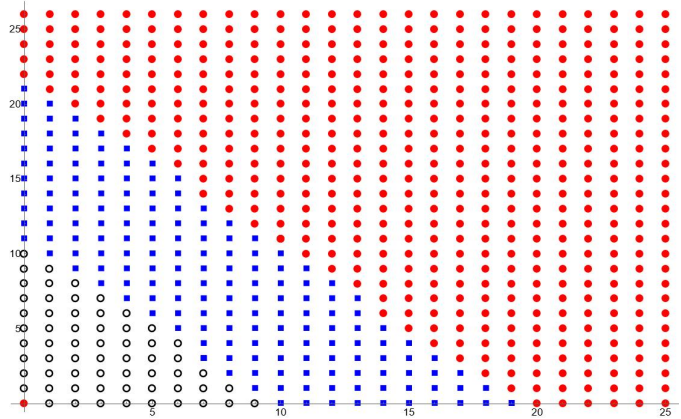}
    \caption{Ordinary $\N^2$-semigroup $S_c$.}
    \label{Ej-ordinario}
\end{figure}
\end{example}

\section{Mult-embedded $\N^2$-semigroups}\label{Embedded}

Given a total order $\preceq$, recall that a $\CaC$-semigroup $S$ is said to be a mult-embedded $\CaC$-semigroup if $S=\{0, m, 2m, \ldots, (k-1)m\}\sqcup S_{km}$ where $m$ is a non-zero element of $\CaC$. The name \textit{mult-embedded} arises from the fact that given any $\CaC$-semigroup, we can always find a mult-embedded $\CaC$-semigroup within it with the same multiplicity. In particular, in the case where $k=1$, an ordinary $\CaC$-semigroup is a mult-embedded $\CaC$-semigroup. 

We continue along the same thread as the previous section. This section is devoted to mult-embedded $\CaC$-semigroups. We focus on the minimal generating set of mult-embedded $\N^2$-semigroups ordered by the graded lexicographic order. Another contribution of this section is to provide a formula for its embedding dimension, which allows us to show that the Generalized Wilf Conjecture and the Extended Wilf Conjecture hold.

We can easily rewrite Lemma \ref{upperboundMSG} for mult-embedded $\N^2$-semigroups ordered by the graded lexicographic order. From now on, we assume that the multiplicity $m$ is not the null vector.
 
\begin{lemma}
Let $k$ be a positive integer, $m\in\N^2$, and $S$ be the mult-embedded $\N^2$-semigroup ordered by $\preceq_{glex}$, which multiplicity is $m$ and conductor equals $km$. If $x\in \N^2$ satisfies $\pi(x)>\pi (\mult_1(S))+\pi (\mult_2(S))$, then $x$ is not a minimal generator.
\end{lemma}
\begin{proof}
    Suppose  $x\in \N^2$ such that $\pi(x)>\pi (\mult_1(S))+\pi (\mult_2(S))$. By the Lemma \ref{upperboundMSG}, since $\msg(S)\subset \msg(S_{km})\cup \{m\}$, it follows that $x\notin  msg(S_{km})$, consequently $x\notin msg(S)$.
\end{proof}

\begin{remark}\label{nota_cota}
    Note that the bound $\pi (\mult_1(S))+\pi (\mult_2(S))$ introduced in the previous lemma depends on the multiplicity $m=(m_1,m_2)$. For the case $m_1=0$, we obtain that $\pi (\mult_1(S))+\pi (\mult_2(S))=2k\pi(m)$, otherwise $\pi (\mult_1(S))+\pi (\mult_2(S))=2k\pi(m)+1$. Hence, any minimal generator $x$ of a mult-embedded $\N^2$-semigroup $S$ satisfies $\pi(x)$ is smaller than or equal to $2k\pi(m)$ or $2k\pi(m)+1$, according to the value of $m$.
\end{remark}

For any $x\in \N^2$, we consider that its coordinates are $x_1$, and $x_2$.

\begin{theorem}\label{thrmEmbedded}
Let $k$ be a positive integer, $m=(m_1,m_2)\in\N^2$, and $S$ be the mult-embedded $\N^2$-semigroup ordered by $\preceq_{glex}$ with multiplicity $m$ and conductor $km$. Then, $S$ is minimally generated by
    \begin{equation} \label{disjointsets}
    \begin{split}
        A= &  \{m\}\sqcup \Big(km,\; (k+1)m\Big)_{\preceq_{glex}} \\
        &  \sqcup \Big\{ x  \in \N^2 \mid (k+1)\pi(m) \leq \pi(x) \leq 2k\pi(m) - 1 \text{ and } x_2 < m_2 \Big\} \\
        &\sqcup \Big\{ x \in \N^2 \mid (k+1)\pi(m) + 1 \leq \pi(x) \leq 2k\pi(m) + 1 \text{ and } x_1 < m_1 \Big\},
    \end{split}
    \end{equation}
    and $\e(S)=\dfrac{(4k-1)\pi^2(m)+\pi(m)+4m_1}{2}$.
    
\end{theorem}
\begin{proof}
    Trivially, $m\in \msg(S)$. Denote $B=\Big(km,\; (k+1)m\Big)_{\preceq_{glex}}$. First, let us prove that $B\subset\msg(S)$. Suppose $x\in B$ such that $x=s+t $ for some $s,t\in S$. We distinguish the following two cases. Without loss of generality, if we consider $\pi(s)\geq k\pi(m)$, then $\pi(t)< k\pi(m)$, and since, $S$ is a mult-embedded $\N^2$-semigroup, it follows that $t=qm$, for some positive integer $q$. Then, $\pi(x)=\pi(s)+\pi(qm)\geq k\pi(m)+q\pi(m)\geq(k+1)\pi(m)$, which contradicts the initial hypothesis.  
    So, we assume $\pi(s), \pi(t)< k\pi(m)$, and then we have that $s$ and $t$ are multiples of $m$. Consequently, $x$ is a multiple of $m$. So, $x\notin B$ and it leads to a false statement. 

    Now, let $x\in \N^2$ such that $ (k+1)\pi(m) \leq \pi(x) \leq 2k\pi(m) - 1 $, and $x_2 < m_2$. If $x=s+t$, for some $s=(s_1,s_2)$ and $t=(t_1,t_2)$ belonging to $S$, then $s_2,t_2<m_2$. It implies  that $s$ and $t$ are not multiples of $m$, and so $\pi(s), \pi(t)\geq k\pi(m)$. Consequently, $\pi(x)\geq2k\pi(m)$. Again, we get a contradiction. Therefore, $x\in \msg(S)$.

    Assuming that $x\in \N^2$ such that $ (k+1)\pi(m)+1 \leq \pi(x) \leq 2k\pi(m) + 1 $, and $x_1 < m_1$, and using a similar structure as above, if $x=s+t$, for some $s,t\in S$, then $s_1,t_1<m_1$. So, $\pi(s), \pi(t)\geq k\pi(m)+1$, and thus, $\pi(x)\geq 2k\pi(m)+2$, which it is not possible. Hence, $x\in \msg(S)$. Summarizing, we have just proved 
    that $A$ is a subset of $\msg(S)$.

From Remark \ref{nota_cota}, $A$ is the minimal generating set of $S$ if and only if no minimal generators belong to the set $\{x\in \N^2\mid \pi(x)\leq 2k\pi(m)\}\setminus A$, when $m_1=0$, or $\{x\in \N^2\mid \pi(x)\leq 2k\pi(m)+1\}\setminus A$, when $m_1\neq 0$.

Consider $x\in\N^2$ such that $x\ge m$ and $x\succeq_{glex} (k+1)m$. Trivially, $(k+1)m\notin \msg(S)$. So, we assume that $x\succ_{glex} (k+1)m$. Since $x\ge m$, $x=m+\lambda$, for some $\lambda\in \N^2$. And thus $\pi (m)+\pi(\lambda)\geq (k+1)\pi (m)$, which implies that $\pi(\lambda)\geq k\pi(m)$. We distinguish two cases depending on the value of $\pi(\lambda)$. If $\pi(\lambda)> k\pi(m)$, then $\lambda\in S$. So, we conclude that $x\notin msg(S)$. If $\pi(\lambda)=k\pi(m)$, and as $x\succ_{glex} (k+1)m$, then $x=((k+1)m_1+i, (k+1)m_2-i)$ for some $i\in \llbracket km_2\rrbracket \setminus\{0\}$. Hence, $x-m=(km_1+i, km_2-i)\in \N^2$, and $\pi (x-m)=k\pi(m)$. We deduce that $x-m\in S$, so $x\notin msg(S)$. 

Finally, suppose $x\in \N^2$, with $x_2<m_2$ such that $\pi(x)=2k\pi(m)+i$ where $i\in\{0,1\}$. 
Note that $x_1+m_2 >x_1+x_2=\pi(x)\ge 2km_1+2km_2$, which implies that $x_1>2k\pi(m)-m_2$. Take $s=(k\pi(m),0)\in S$. Since $x-s=(x_1-k\pi(m), x_2)\in \N^2$, and $\pi(x-s)=k\pi(m)+i\geq k\pi(m)$, we obtain that $x-s$ belongs to $S$. Whence, $x$ is not a minimal generator.

By definition, $\e(S_c)=\sharp A$. Similarly to the proof of Lemma \ref{propOr2}, 

\begin{align*}
    \sharp A&= 1+km_2+\sum_{l=k\pi(m)+1}^{(k+1)\pi(m)-1}(l+1)+(k+1)m_1\\
    & + m_2(k-1)\pi(m) +m_1\big((k-1)\pi(m)+1\big)  \\
       & =\dfrac{(4k-1)\pi^2(m)+\pi(m)+4m_1}{2}.
\end{align*}
\end{proof}

The following corollary is a direct consequence of the Theorem \ref{thrmEmbedded}.

\begin{corollary}
    Every mult-embedded $\N^2$-semigroup, ordered by $\preceq_{glex}$, satisfies the Generalized  Wilf Conjecture and Extended Wilf Conjecture.
\end{corollary}
\begin{proof}
     Let $S$ be a mult-embedded $\N^2$-semigroup with conductor $km = k(m_1, m_2)$, ordered by $\preceq_{glex}$. Given Proposition \ref{propAUX}, it is enough to show that the inequality \eqref{InewilfCisto} holds.   Trivially, $\nu(S)=k$, and we assume that $k>1$, otherwise, it has already been proved (Corollary \ref{coroORWilf}). If $c_1=0$ then $\gamma(S)=km_2$. Thus,
     \[
     ke(S)=k\frac{(4k-1)m_2^2 + m_2 }{2} \geq  k\frac{((4k-1)m_2 + 1)m_2 }{2} \geq 2\gamma(S),
     \]
     which it is true, since $(4k-1)m_2+1\geq 4$ for all $k>1$, and natural number $m_2$.
     If $c_1\ne 0$, then $\gamma(S)=km_1(km_2+2)$. According to Theorem \ref{thrmEmbedded}, 
     \begin{align*}
     ke(S)&=k\frac{(4k-1)\pi^2(m) + \pi(m) + 4m_1}{2} \geq k\frac{3k \pi^2(m)+5m_1}{2}\\
     &\geq \frac{4k^2m_1m_2 + 3k^2m_1^2+5km_1}{2}\geq  \frac{4k^2m_1m_2 + 8km_1}{2}=2\gamma(S),
\end{align*}
    which proves the result.
\end{proof}

The section concludes with an example to illustrate the concepts and results discussed.

\begin{example}
    Let $S$ be a mult-embedded $\N^2$-semigroup ordered by $\preceq_{glex}$, with multiplicity $m = (4,2)$, and conductor $3m=(12,6)$,  which is shown in Figure \ref{Ej-embedded}. As mentioned earlier, the empty circles are the gaps of $S$, the blue squares are the minimal generators of $S$, and the red circles are elements of $S$.
\begin{figure}[ht]
    \centering
    \includegraphics[scale=.35]{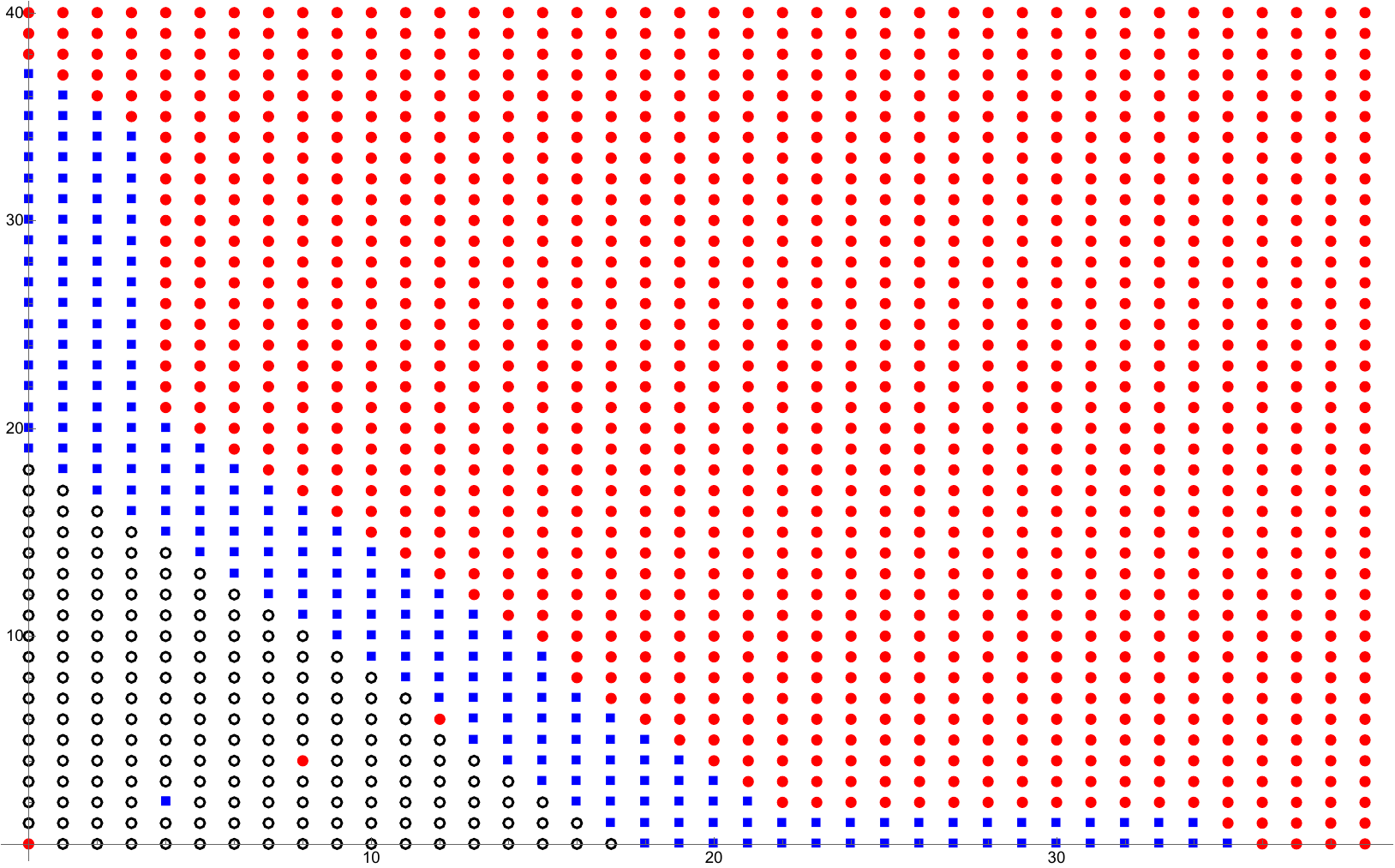}
    \caption{Mult-embedded $\N^2$-semigroup.}
    \label{Ej-embedded}
\end{figure}

From Proposition \ref{propOr2}, we obtain that $\e(S) = \frac{396 + 6 + 16}{2}=209 $, and 
\begin{equation*}
    \begin{split}
        \msg(S)&=   
        \{(4,2)\}\sqcup \Big((12,6),\; (16,8)\Big)_{\preceq_{glex}}\\
        \sqcup  &\bigcup_{i=0}^{11} \Big[(23+i,1), \;(24+i,0)\Big]_{\preceq_{glex}} \sqcup  \bigcup_{i=0}^{12} \Big[(0,25+i), \;(3, 22+i)\Big)_{\preceq_{glex}}.
    \end{split}
    \end{equation*}

\end{example}

\subsection*{Funding}

The last author is partially supported by grant PID2022-138906NB-C21 funded by MICIU/AEI/ 10.13039/501100011033 and by ERDF/EU.

Consejería de Universidad, Investigación e Innovación de la Junta de Andalucía project ProyExcel\_00868 and research group FQM343 also partially supported all the authors.

This publication and research have been partially granted by INDESS (Research University Institute for Sustainable Social Development), Universidad de Cádiz, Spain.

\subsection*{Author information}
J.C. Rosales. Departamento de \'{A}lgebra, Universidad de Granada, E-18071 Granada, (Granada, Spain).
E-mail: jrosales@ugr.es.

\medskip

\noindent
R. Tapia-Ramos. Departamento de Matem\'aticas, Universidad de C\'adiz, E-11406 Jerez de la Frontera (C\'{a}diz, Spain).
E-mail: raquel.tapia@uca.es. 

\medskip

\noindent
A. Vigneron-Tenorio. Departamento de Matem\'aticas/INDESS (Instituto Universitario para el Desarrollo Social Sostenible), Universidad de C\'adiz, E-11406 Jerez de la Frontera (C\'{a}diz, Spain).
E-mail: alberto.vigneron@uca.es.

\subsection*{Conflict of Interest}
The authors declare no conflict of interest.

\end{document}